\documentclass{amsart}

\usepackage[english]{babel} 
\usepackage{exscale}   
\usepackage{amsmath}
\usepackage{amsthm}
\usepackage{amsfonts}
\usepackage{amssymb}
\usepackage{amsxtra}
\usepackage{xspace}
\usepackage{epsfig}
\usepackage{latexsym}
\usepackage{graphics}
\usepackage{xcolor}
\usepackage{mathabx}

\usepackage{mathtools}
\usepackage{tikz}

\usepackage{tikz}
\usepackage{tikz-cd}
\usetikzlibrary{cd}

\usepackage{comment}

\newtheorem{theorem}{Theorem}[section]

\newtheorem{proposition}[theorem]{Proposition}
\newtheorem{corollary}[theorem]{Corollary}
\newtheorem{lemma}[theorem]{Lemma}
\newtheorem{definition}[theorem]{Definition}
\newtheorem{example}[theorem]{Example}

\newtheorem{prop}[theorem]{Proposition}
\newtheorem{defn}[theorem]{Definition}
\newtheorem{rem}[theorem]{Remark}
\newtheorem{rems}[theorem]{Remarks}
\newtheorem{examples}[theorem]{Examples}

\newcommand{\1}{\mbox{$1\!\!\!\:{\rm I}$}}

\newcommand{\BH}{\mbox{$\mathcal{B}(\mathcal{H})$}}

\newcommand{\Bset}{\mathbb{B}}
\newcommand{\Cset}{\mathbb{C}}

\newcommand{\Nset}{\mathbb{N}}

\newcommand{\Sset}{\mathbb{S}}

\newcommand{\Zset}{\mathbb{Z}} 

\newcommand{\CC}{\ensuremath{{\mathcal C}}\xspace}         
\newcommand{\CD}{\ensuremath{{\mathcal D}}\xspace}         
\newcommand{\CH}{\ensuremath{{\mathcal H}}\xspace}         
\newcommand{\CK}{\ensuremath{{\mathcal K}}\xspace}         

\newcommand{\NN}{\mathbb{N}}
\newcommand{\ZZ}{\mathbb{Z}}

\newcommand{\Ff}{\mathcal{F}}
\newcommand{\Hh}{\mathcal{H}}

\newcommand{\e}{\varepsilon}

\newcommand{\cat}{\Lambda}

\newcommand{\diff}{\partial}
\newcommand{\into}{\rightarrow}
\newcommand{\ip}[1]{\langle #1 \rangle}

\renewcommand{\root}[1]{\underline{#1}}
\newcommand{\ssc}{\Delta_{S}}

\DeclarePairedDelimiter{\card}{\lvert}{\rvert}
\DeclarePairedDelimiter{\rank}{\lvert}{\rvert}

\newcommand{\lev}[1]{\ell(#1)}

\DeclareMathOperator{\id}{id}
\DeclareMathOperator{\im}{im}

\DeclareMathOperator{\spn}{span}

\DeclareMathOperator{\solim}{sot-lim}

\tikzset{vertex/.style={inner sep=1pt, circle, fill=black}} 
\tikzset{lvertex/.style={inner sep=0.5pt, minimum width=10pt, circle,draw}} 
\tikzset{cvertex/.style={inner sep=0.5pt, circle, draw=red}} 
\tikzset{dedge/.style={-latex}} 
\tikzset{tdedge/.style={red, thick, -latex}} 
\tikzset{rdedge/.style={red,-latex}}
\tikzset{bdedge/.style={blue,-latex}}
\tikzset{gdedge/.style={green,-latex}}
\tikzset{cylch/.style={circle,blue,thick,draw,text=black,inner sep=1pt}}
\tikzset{colind/.style={->,blue,thick}}
\tikzset{colindl/.style={circle,draw,font=\tiny,inner sep=2pt,above}}
\tikzset{deltas/.style={midway, sloped, above, font=\tiny}}

\begin{document}

\title{Semi-Cosimplicial Hilbert Spaces with Isometric Coface Operators}

\author{D.~Gwion Evans}
\address{D.~Gwion Evans \\
Department of Mathematics \\ 
Aberystwyth University \\
Aberystwyth, SY23 3BZ, UK}
\email{dfe@aber.ac.uk}
\author{Rolf Gohm}
\address{Rolf Gohm \\
Department of Mathematics \\ 
Aberystwyth University \\
Aberystwyth, SY23 3BZ, UK}
\email{rog@aber.ac.uk}
\author{Claus K\"{o}stler}
\address{Claus K\"{o}stler\\
School of of Mathematical Sciences \\ 
University College Cork\\
Cork, Ireland}
\email{Claus@ucc.ie}

\begin{abstract}
Semi-cosimplicial objects in the category of Hilbert spaces with isometries which are motivated by non-commutative probability theory, in particular by the distributional symmetry of spreadability, are introduced and systematically developed in various directions: partial shifts, cohomology, Hessenberg form, a related graph, decomposition into labeled subspaces, representation theory of the infinite symmetric and braid groups, classification and extensions for semi-cosimplicial sets with injective coface maps and a toy version of spreadability. 
\end{abstract}

\maketitle

\section{Introduction}\label{S:intro}

In \cite{EGK17} we explained how the distributional symmetry of spreadability, in particular if we include the wider area of non-commutative probability, can be understood algebraically by considering semi-cosimplicial objects in the category of (non-commutative) probability spaces. We remind the reader that the semi-simplicial category $\Delta_S$ has as objects finite ordered sets, usually written as $[n] := \{0,1,\ldots,n\},\; n \in \Nset_0$, and the morphisms are all strictly increasing maps between these sets.
A covariant functor $F$ from the semi-simplicial category $\Delta_S$ to another category $\CC$ is called a semi-cosimplicial object (SCO) in $\CC$. We can work out a more explicit description of what a SCO is by noting that the morphisms of $\Delta_S$ are generated by the face maps
\[
\e_k \colon [n-1] \rightarrow [n],\quad
m \mapsto m \; \text{if}\;\; m<k,\; 
m \mapsto m+1 \; \text{if}\;\; m \ge k.
\]
Here $k=0,\ldots,n$ and $n \in \Nset$. Following the usual convention we omit the index $n$ in the notation
of the $\e_k$ and leave the domain and codomain to the context. The $\e_k$ satisfy the cosimplicial identities
\[
\e_j \e_i = \e_i \e_{j-1} \quad \text{if}\,\;i < j
\]
and these cosimplicial identities provide a presentation of the category $\Delta_S$. The functor $F$ takes $[n]$ to $F[n]$ and $\e_k$ to $\delta_k=F(\e_k): F[n-1] \rightarrow F[n]$. Hence we can say that
a semi-cosimplicial object (SCO) in the category $\CC$ is a sequence $(F_n)_{n\in \Nset_0}$ of objects in $\CC$ together with morphisms (coface operators)
\[
\delta_k: F_{n-1} \rightarrow F_n \quad\quad (k=0,\ldots,n)
\]
satisfying the cosimplicial identities
\[
\delta_j \delta_i = \delta_i \delta_{j-1} \quad \text{if}\,\;i < j\,.
\]
We refer to, for example, \cite{We94} for more information about (semi-co)simplicial objects and for a development of the rich and far developed theory built around them. 

As mentioned above, in \cite{EGK17} we have been mainly interested in SCOs in the category of (non-commutative) probability spaces and we are led to some deep questions, for example if there is always a representation of the braid group in the background (framed as braidability in \cite{GK09}). To make some progress about such questions we decided to investigate SCOs in a category which is easier but closely related to the category of (non-commutative) probability spaces, namely the category of Hilbert spaces. For example by applying the GNS-construction we can go from the former to the latter. From this motivation it is natural to use isometries as morphisms of Hilbert spaces. Another wider choice would have been contractions as morphisms. But it turned out that the more rigid choice of isometries leads to a very rich and interesting structure worth to consider on its own. This is what we want to do in this paper. It is impossible in a single paper to follow all the threads to their ends. So the philosophy is to establish in each section some basic direction and some results about it and to end with open questions and perspectives. Let us give short summaries of these topics. 

In Section 2 we define semi-cosimplical Hilbert spaces (SCHs) to be augmented SCOs in the category of Hilbert spaces with isometries, as motivated above, and establish some basic properties. We introduce partial shifts. A main result is the toy de Finetti theorem connecting saturatedness of the SCH with a shift property of the partial shifts. 

In Section 3 we look at the cohomology of SCHs. A basic observation is that the cohomology of a saturated SCH is always trivial. 

In Section 4 we establish a way to write partial shifts in a Hessenberg form which allows us to study how they can be obtained from unitary representations of the infinite braid group $\Bset_\infty$. This leads to the definition of braided SCHs. 

In Section 5 we introduce a category $\Lambda_S$ closely related to the semi-simplicial category $\Delta_S$ and the corresponding graph which visualises the dynamics in a SCH. 

In Section 6 we construct a decomposition of the SCH into what we call labeled subspaces. One of the main results is that braided saturated SCHs are characterized by the orthogonality of labeled subspaces. We give more characterizations and call these SCHs normal. 

In Section 7 we explain that the notion of a normal SCH really could have been invented as the structure of tame representations of the infinite symmetric group $\Sset_\infty$. We give an alternative proof of the main theorem about these, Lieberman's theorem, by our SCH methods. 

In Section 8 we address the question of extending non-normal to normal SCHs. We prove that for SCSs (semi-cosimplicial sets with injective coface maps) this is always possible and derive a classification. 

In Section 9 we define some toy version of spreadability which can be studied and classified by SCHs. 

More work needs to be done to obtain the full circle back to the motivating problems in non-commutative probability theory. 

\section{SCHs, Partial Shifts, Saturation and a Toy de Finetti Theorem}\label{sec:defs}

For $n \in \ZZ$ we let $\NN_n:=\{m \in \ZZ : m \ge n\}$ and define $[n]:=\{m \in \NN_0 : m \le n\}$, so that, in particular, $[-1]=\emptyset$.

\begin{definition}
A semi-cosimplicial Hilbert space $(\CH_k, \delta_i)_{k,i}$
(abbreviated: SCH) is a sequence $(\CH_k)_{k \in \NN_{-1}}$ of Hilbert spaces
together with isometries $\delta_i: \CH_{n-1} \rightarrow \CH_n$ 
(for all $n \in \Nset_0$ and $i=0,\ldots,n$) which satisfy
$\delta_j \delta_i = \delta_i \delta_{j-1}$ if $0 \le i < j \le n$.
\end{definition}
Note that $\delta_i$ also depends on $n$ but, as common in homological algebra, 
this is suppressed in the notation.

\begin{definition}
An SCH $(\widecheck{\CH}_k,\widecheck{\delta}_i)_{k,i}$ is a  sub-SCH of the SCH $(\CH_k, \delta_i)_{k,i}$ if $\widecheck{\CH}_k \subset \CH_k$ for all $k \ge -1$ and $\widecheck{\delta}_i = \delta_i |_{\widecheck{\CH}_{n-1}}$ (for all $n \in \Nset_0$ and $i=0,\ldots,n$).
\end{definition}

\begin{definition}\label{D:partial_shifts}
A sequence of isometries $(\alpha_n)_{n \in \Nset_0}$, with $\alpha_n: \CH \rightarrow \CH$ where $\CH$ is a Hilbert space, is called a sequence of partial shifts if $\alpha_j \alpha_i = \alpha_i \alpha_{j-1}$ for all $0 \le i < j$.	
\end{definition}	
This is related to SCHs as follows. Given an SCH
$(\CH_k, \delta_i)_{k,i}$ we interpret $\delta_n: \CH_{n-1} \rightarrow \CH_n$ as an embedding (for all $n \in \Nset_0$), hence we obtain a tower $\CH_{-1} \subset \CH_0 \subset \CH_1 \subset \ldots$. We can define a Hilbert space
$\CH_\infty$ as the closure of ${\bigcup_{k \ge -1}{\CH_k}}$
and maps $(\alpha_n)_{n \in \Nset_0}$ determined by $\alpha_n |_{\CH_k} := \delta_n |_{\CH_k}$ for all $k \ge n-1$. These are well defined on $\CH_\infty$ because this is compatible with the embeddings: if $\CH_k \ni x = \delta_{k+1}(x) \in \CH_{k+1}$ then, for all $k \ge n-1$, we have
$\alpha_n \delta_{k+1}(x) = \delta_n \delta_{k+1}(x) =
\delta_{k+2} \delta_n(x) = \delta_n(x) = \alpha_n(x)$ (because $\delta_n(x) \in \CH_{k+1}$, so 
$\delta_{k+2}$ is just an embedding).
Hence the $\alpha_n$ extend by continuity to isometries on $\CH_\infty$ and, inheriting the properties of the $\delta_n$, they form a sequence of partial shifts which we refer to as associated to the SCH.
Categorically this is a colimit (compare \cite{EGK17}). 

By definition $\alpha_n$ acts identically on $\CH_{n-1}$ (for all $n \in \Nset_0$) which we also write as $\CH_{n-1} \subset \CH_\infty^{\alpha_n}$ (the latter being the notation for the fixed point space of $\alpha_n$). For all $k \ge -1$ we have $\alpha_n (\CH_k) \subset \CH_{k+1}$ and we refer to this property as adaptedness with respect to the tower.

Conversely, if we start with a sequence of partial shifts $(\alpha_n)_{n \in \Nset_0}$ on a Hilbert space $\CH$ then from any tower $(\CH_k)_{k \ge -1}$ to which the $\alpha_n$ are adapted we get an SCH by defining coface operators $\delta_i: \CH_{n-1} \rightarrow \CH_n$ as $\delta_i := \alpha_i|_{\CH_{n-1}}$. If in addition $\CH_{n-1} \subset \CH^{\alpha_n}$ is valid for all $n$ then the sequence of partial shifts obtained by restricting the original sequence to (the invariant subspace) $\CH_\infty$ is associated to this SCH.

For later use we mention the natural notion of unitary equivalence for SCHs.
\begin{definition}
Two SCHs $(\CH_k, \delta_i)_{k,i}$ and $(\CH^\prime_k, \delta^\prime_i)_{k,i}$ are unitarily equivalent (isomorphic) if for all $k$ there are unitaries $U_k: \CH_k \rightarrow \CH^\prime_k$ such that $\delta^\prime_i U_{n-1} = U_n \delta_i$ (for all $n \in \Nset_0$ and $i=0,\ldots,n$). Or, equivalently, if there is a unitary $U$ from
$\CH_\infty$ to $\CH^\prime_\infty$ such that $U \CH_k = \CH^\prime_k$ for all $k$ and $\alpha^\prime_i U = U \alpha_i $ for all $i \in \Nset_0$.
\end{definition}

We see that it is rather a matter of convenience if we want to state results in terms of SCHs or of partial shifts together with an adapted tower $(\CH_k)_{k \ge -1}$ such that $\CH_{n-1} \subset \CH^{\alpha_n}$ for all $n$. Indeed we will often just choose the more convenient version in the following.

 
\begin{definition}
An SCH $(\CH_k,\delta_i)$ is called saturated if $\CH_{n-1} = \CH_\infty^{\alpha_n}$ for all $n \in \Nset_0$, where $(\alpha_n)_{n \in \Nset_0}$ is the associated sequence of partial shifts. 	
\end{definition}



\begin{proposition}\label{P:saturation}	
Any sequence $(\alpha_n)_{n \in \Nset_0}$ of partial shifts is associated to the saturated SCH $(\CH_k, \delta_i)_{k,i}$ given by $\CH_{k} := \CH^{\alpha_{k+1}}$
for all $k$.
\end{proposition}

\begin{proof}
Given a sequence $(\alpha_n)_{n \in \Nset_0}$ of partial shifts we put 
$\CH_{n-1} := \CH^{\alpha_n}$ for all $n \in \Nset_0$. Fix $m \in \NN_0$ and suppose $x \in \CH^{\alpha_m}$. Then $\alpha_{m+1}(x) = \alpha_{m+1}\alpha_m(x) = \alpha^2_m(x) = x$.
Hence $x \in \CH^{\alpha_{m+1}}$ and, by induction, $x \in \CH^{\alpha_n}$ for all $n \ge m$. On the other hand, if $n < m$ then $\alpha_{m+1} \alpha_n(x) = \alpha_n \alpha_m(x) = \alpha_n(x)$ and hence  $\alpha_n(x) \in \CH^{\alpha_{m+1}}$. This proves adaptedness. 
\end{proof}

\begin{corollary}\label{C:saturation}
Every SCH $(\CH_k, \delta_i)_{k,i}$ is a sub-SCH of a saturated SCH, namely the saturated SCH given in Proposition \ref{P:saturation} for the associated partial shifts.
\end{corollary} 

\begin{definition}
  Given an SCH $(\CH_k, \delta_i)_{k,i}$,  the saturated SCH in Corollary \ref{C:saturation} is called the saturation of $(\CH_k,\delta_i)_{k,i}$; we denote it by $(\hat{\CH}_k,\hat{\delta}_i)_{k,i}$.
\end{definition}

\begin{rems}
   The saturation $(\hat{\CH}_k,\hat{\delta}_i)_{k,i}$ of an SCH
   $(\CH_k,\delta_i)_{k,i}$:
   \begin{enumerate}
  \item has the same ambient Hilbert space, i.e.
    $\CH_\infty:=\overline{\bigcup_{k \ge
        -1}{\CH_k}}=\overline{\bigcup_{k \ge
        -1}\hat{\CH}_k}=:\hat{\CH}_\infty$, and the same associated
    sequence of partial shifts
    $(\alpha_n)_{n \in \Nset_0}=(\hat{\alpha}_n)_{n \in \Nset_0}$;
  \item is the smallest saturated SCH containing $(\CH_k,\delta_i)_{k,i}$ as a sub-SCH, in the sense that it is a sub-SCH of any other saturated SCH that contains $(\CH_k,\delta_i)_{k,i}$ as a sub-SCH.
  \end{enumerate}

\end{rems}

\begin{definition}
Given an SCH $(\CH_k, \delta_i)_{k,i}$we call
\[
\CD_{-1} := \CH_{-1},\; \CD_n := \CH_n \ominus \CH_{n-1} \;\; \text{for} \; n \in \Nset_0
\]	
the innovation spaces of the SCH.
\end{definition}
\begin{rem}
We have that $\CH_\infty := \overline{\bigcup_{k \ge -1}{\CH_k}} = \bigoplus_{k \ge -1} \CD_k$.
\end{rem}

It will be useful for us to consider localised versions of saturation.

\begin{definition}
  We say that an SCH $(\CH_k,\delta_i)_{k,i}$ is
  \begin{enumerate}
  \item saturated at level $n$ if
    $\CH_n=\hat{\CH}_n:=\CH_\infty^{\alpha_{n+1}}$;
    \item saturated at level $n$ up to level $m$ if $\CH_n=\hat{\CH}_n\cap \CH_m$.
\end{enumerate}
\end{definition}

For the remainder of this section $(\CH_k,\delta_i)_{k,i}$ will denote an arbitrary SCH with sequence of innovation spaces $(\CD_k)_k$. Its saturation will be denoted by $(\hat{\CH}_k, \hat{\delta}_i)_{k,i}$; the sequence of innovation spaces of the saturation will be denoted by $(\hat{D}_k)_k$. For each $n \in \NN_{-1}$, we let $P_n$ and $\hat{P}_n$ be the orthogonal projections on $\CH_\infty=\hat{\CH}_\infty$ onto $\CH_n$ and $\hat{\CH}_n$, respectively.

\begin{lemma}\label{L:cosimplicial_ids_for_proj}
  For each $i,n \in \NN_0$ with  $i \le n$, we have $\hat{P}_n\alpha_i=\alpha_i\hat{P}_{n-1}$.
\end{lemma}

\begin{proof}
 Recalling that for each $n \in \NN_0$, $\hat{\CH}_{n-1}=\CH_\infty^{\alpha_{n}}$, we have, by von Neumann's mean ergodic theorem, $\hat{P}_{n-1}=\solim_{N \to \infty}\frac{1}{N} \sum_{m=0}^{N-1} \alpha^m_{n}$. It now follows from the cosimplicial identities that $\hat{P}_n\alpha_i=\alpha_i\hat{P}_{n-1}$ for each $0\le i \le n$.
\end{proof}

\begin{theorem}(`toy de Finetti')\label{T:saturation}
  Let $(\CH_k,\delta_i)_{k,i}$ be an SCH. For each $n \in \NN_0$,
  \begin{align*}
    & \alpha_n(\CD_k) \subset \CD_{k+1} \text{ for all } k \ge n \\ 
    \implies & (\CH_k,\delta_i)_{k,i}\text{ is saturated at level $n-1$} \\
    \implies & \alpha_i(\CD_n)\subset \CD_{n+1} \text{ for all } 0\le i \le n.
  \end{align*}
  The following assertions are equivalent:
  \begin{enumerate}
      \item The SCH $(\CH_k,\delta_i)_{k,i}$ is saturated.
      \item $\alpha_i(\CD_n)\subset \CD_{n+1}$ for all $i,n \in \Nset_0$ such that $0 \le i \le n$.
      \item $\alpha_n(\CD_n)\subset \CD_{n+1}$ for all $n \in \Nset_0$.
  \end{enumerate}
\end{theorem}

\begin{proof}
  Let $n \in \NN_0$. Suppose that $\alpha_n(\CD_k) \subset \CD_{k+1}$ for all $k \ge n$. For $x \in \CH_\infty^{\alpha_n}$ write $x=y+z$ where $y\in \CH_{n-1}$ and $z\in \CH_\infty^{\alpha_n} \ominus \CH_{n-1}$ (noting $y\in \CH_{n-1} \subset \CH_\infty^{\alpha_n}$). As $z \in \bigoplus_{j\ge n } \CD_{j} \cap \CH_\infty^{\alpha_n}$ we see that $z=\alpha_n(z)$ implies $z=0$ (since from $n+1$ onwards, each component of $z$ is equal to the image of its predecessor under the partial shift, and the component in $\CD_n$ is zero). Hence $x \in \CH_{n-1}$. This shows saturatedness at level $n-1$.
  Now suppose that $\CH_{n-1}=\CH_{\infty}^{\alpha_n}$, 
  so $P_{n-1}=\hat{P}_{n-1}$. 
  For each $0 \le i \le n$ and $x \in \CD_n$ (hence $P_{n-1}(x)=0$) we have, by Lemma \ref{L:cosimplicial_ids_for_proj},
  \[ P_n \alpha_i(x) = P_n \hat{P}_n \alpha_i(x) = P_n \alpha_i \hat{P}_{n-1}(x)=P_n\alpha_iP_{n-1}(x) = 0.\] Thus $\alpha_i(x) \in \CH_{n+1}\ominus \CH_{n}=\CD_{n+1}$. Hence $\alpha_i(\CD_n)\subset \CD_{n+1}$. We have proved the implications for the localised statements. But now we observe that the final statement valid for all $n \in \Nset_0$ is identical to the first statement for all $n \in \Nset_0$, so we immediately get the equivalence of (1) and (2). It is obvious that (2) implies (3).
 We can deduce (2) from (3) from a localised version which we state in Lemma \ref{L:one4all}. 
\end{proof}

  The converses of the localised implications in Theorem \ref{T:saturation} do not hold in general, as demonstrated in Examples \ref{E:basic_examples} below.
As an explanation why we consider Theorem \ref{T:saturation} as a toy version of the probabilistic de Finetti theorem classifying spreadable sequences of random variables we recommend to look at Theorem \ref{C} below where it is used to obtain a classification of a Hilbert space version of spreadability.

We observe in the following lemma that it is sufficient to check that the $n^\text{th}$ partial shift shifts the $n^\text{th}$ innovation space to ensure that all lower partial shifts do so also. In particular, we use this observation when considering the cohomology of an SCH in Section \ref{S:cohomology}.

\begin{lemma}\label{L:one4all}
  
For each $n\in \NN_0$, if $\alpha_n(\CD_n) \subset \CD_{n+1}$ then $\alpha_i(\CD_n) \subset \CD_{n+1}$ for all $0\le i \le n$.
    \end{lemma}

    \begin{proof}
      
      For $n \in \NN_0$, suppose that $\alpha_{n}(\CD_n) \subset \CD_{n+1}$. By Lemma \ref{L:cosimplicial_ids_for_proj}, for $x \in \CD_n, i \in \NN_0$ with $i \le n$, we have $P_n\alpha_i(x) = P_n\hat{P}_n\alpha_i(x) = P_n\alpha_i\hat{P}_{n-1}(x)$. We claim that $\hat{P}_{n-1}(x)=0$. To see why, given that $x \in \CD_n \subset \CH_n \subset \hat{\CH}_n$ we may write $x=y+z$ where $y=\hat{P}_{n-1}(x), z \in \hat{D}_n$. As $\CD_{n} \subset \CD_{n+1}^\perp$, $y \in \CH^{\alpha_n}_\infty$ and $\ip{\alpha_n(y),\alpha_n(z)}=\ip{y,z}=0$, we have (now making use of the assumption $\alpha_{n}(\CD_n) \subset \CD_{n+1}$)
      \begin{align*}
        0&= \ip{x,\alpha_n(x)} = \ip{y+z, \alpha_n(y) + \alpha_n(z)} = \ip{y,y} + \ip{z,\alpha_n(z)}.
      \end{align*}
      Furthermore, by Theorem \ref{T:saturation} we have $\alpha_n(\hat{D}_n)\subset \hat{D}_{n+1}$, thus $\ip{z,\alpha_n(z)}=0$. Hence $\hat{P}_{n-1}(x)=y=0$ so that $\alpha_i(x) \in \CD_{n+1}$.
    \end{proof}
    
We finish this section with an elementary way to produce examples.
\begin{definition}
A 
semi-cosimplicial set $(X_k, \delta_i)_{k,i}$
(abbreviated: SCS) is a sequence $(X_k)_{-1 \le k \in \Zset}$ of sets
together with injective maps $\delta_i: X_{n-1} \rightarrow X_n$ 
(for all $n \in \Nset_0$ and $i=0,\ldots,n$) which satisfy
$\delta_j \delta_i = \delta_i \delta_{j-1}$ if $0 \le i < j \le n$.
\end{definition}
Given an SCS $(X_k, \delta_i)_{k,i}$ we can build an SCH $(\CH_k, \delta_i)_{k,i}$ by interpreting the elements of the sets $X_k$ as orthonormal basis vectors of $\CH_k$ (for all $k$) and by extending the maps $\delta_i$ linearly to isometries (which we still call $\delta_i$). All the definitions and statements about SCHs have natural analogues for SCSs. For example we have a tower $X_{-1} \subset X_0 \subset \ldots$, innovation sets $D_k := X_k \setminus X_{k-1}$ and partial shifts $\alpha_i$ on $X_\infty = \bigcup_k X_k$ respectively on $\CH_\infty = \ell^2(X_\infty)$. Furthermore, the SCS is saturated at level $n$ if $X_{n}=X_\infty^{\alpha_{n+1}}$ and saturated at level $n$ up to level $m$ if $X_n=X_\infty^{\alpha_{n+1}}\cap X_m$. We say in this situation that the SCH has an underlying SCS. 

\begin{examples}\label{E:basic_examples}
  \begin{enumerate}
  \item \label{E:proto} The prototypical example is given by the SCS $(X_k,\delta_i)_{k,i}$ with $X_k:=[k]$ for all $k \in \NN_{-1}$ and $\delta_i:=\e_i$ for all $i$ (as in Section 1).  We have $X_\infty=\NN_0$ and the partial shifts satisfy
    \[ \alpha_i(n)=
    \begin{cases}
      n+1 & \text{ if } 0\le i \le n, \\ n & \text{ if } i>n.
    \end{cases}
  \]
  This example explains the terminology `partial shift'.
\item A modification of the prototypical example: start with $\NN_0$ and the same sequence of partial shifts $(\alpha_i)_i$ as for the prototypical example above.  Let $X_{-1}=X_0=X_1=\emptyset$, $X_2=\{0,2\}$, $X_3=X_2\cup\{1,3\}$ and let $X_n=X_{n-1}\cup\{n\}$ for all $n \ge 4$. As $X_{n-1}\subset \NN_0^{\alpha_n}$, we may form the SCS $(X_k,\delta_i)_{k,i}$ (with each $\delta_i$ being a restriction of $\alpha_i$) (cf. discussion after Definition \ref{D:partial_shifts}).  Notice this example illustrates that the converse to the first implication in Theorem \ref{T:saturation} does not hold, since
  $\CH_0=\CH_{\infty}^{\alpha_1}=\{0\}$ but $\alpha_0(D_3) \not\subset D_4$ (since $1 \in D_3$ but $\alpha_0(1)=2\in D_2$ so $\alpha_0(1)\notin D_4$).

  Note that with this method of `putting the balls $n$ into the wrong boxes $D_{\ell(n)}$' we can produce many examples. We can check that exactly if $\ell: \NN_0 \rightarrow \NN_0$ satisfies $n \le \ell(n) \le \ell(n-1)+1$ for all $n \ge 1$ (where $\ell(0)$ is arbitrary) it gives us an SCS which is associated to the prototypical partial shifts on $\NN_0$. The choice $\ell(n)=n$ for all $n$
  is saturated (right boxes), all the others are not (wrong boxes). The example above is obtained from $\ell(0)=2,\, \ell(1)=3$ and $\ell(n)=n$ for all $n \ge 2$. It is sketched in the figure below.

\begin{figure}[!h]
  \begin{tikzpicture}
 \node (D-1) at (0,1) {$D_{-1}$};
 \node (D0) at (1,1) {$D_0$};
 \node (D1) at (2,1) {$D_1$};
 \node (D2) at (3,1) {$D_2$};
 \node (D3) at (4,1) {$D_3$};
 \node (D4) at (5,1) {$D_4$};
 \node at (6,1) {$\dots$};
 \node at (6,-1.4) {$\dots$};
 
 \pgfmathtruncatemacro{\b}{-2}
 
  \draw [dashed,gray] (0,\b) -- (D-1);
  \draw [dashed,gray] (1,\b) -- (D0);
  \draw [dashed,gray] (2,\b) -- (D1);
  \draw [dashed,gray] (3,\b) -- (D2);
  \draw [dashed,gray] (4,\b) -- (D3);
  \draw [dashed,gray] (5,\b) -- (D4);
 
  \node [vertex] (0) at (3,0) {};
  \node [vertex] (1) at (4,-0.7) {};
  \node [vertex] (2) at (3,-1.4) {};
  \node [vertex] (3) at (4,-1.4) {};
  \node [vertex] (4) at (5,-1.4) {};

\node at (0) [left] {0};
\node at (1) [right] {1};  
\node at (2) [below] {2};
\node at (3) [below] {3};
\node at (4) [below] {4};

  
  \draw [dedge] (0) -- (1);
  \draw [dedge] (1) -- (2);
  \draw [dedge] (2) -- (3);
  \draw [dedge] (3) -- (4);

\end{tikzpicture}
\label{F:examplesSCS}

\end{figure}

\item Let $\CH_{-1}=\CH_{0}=\{0\}$ and let $\CH_n=\Cset$ for all $n > 0$. Let $\alpha_n=\id$ for all $n \in \NN_0$. Then $(\alpha_n)_n$ is a sequence of partial shifts, which induces an SCH.  This trivial example illustrates that the converse to the second implication in Theorem \ref{T:saturation} does not hold, since $\alpha_0(\CD_0) \subset \CD_1$ but $\CH_\infty^{\alpha_0}=\Cset \neq \CH_{-1}$.
\end{enumerate}
\end{examples}

In the case of SCSs we can strengthen Theorem \ref{T:saturation} as follows.

\begin{theorem}\label{T:scs-saturation}
  Fix $n \in \NN_0$. The SCS $(X_k,\delta_i)_{k,i}$ is saturated at level $n-1$ up to level $n$ if and only if $\alpha_i(D_n)\subset D_{n+1}$ for all $0\le i \le n$.
\end{theorem}

\begin{proof}
  Suppose that the SCS $(X_k,\delta_i)_{k,i}$ is saturated at level $n-1$ up to level $n$ for some $n \in \NN_0$. Then for each $0\le i \le n$ and $x \in D_n$, suppose $\alpha_i(x) \in X_n$. Then
  \[ \alpha_i(x) = \alpha_{n+1} \alpha_i(x) = \alpha_i\alpha_n(x)\]
  so that $x=\alpha_n(x)$ since $\alpha_i$ is injective.  Hence $x \in X^{\alpha_n} \cap X_n = X_{n-1}$, which is a contradiction since $D_n\cap X_{n-1}=\emptyset$. Therefore we must have $\alpha_i(x) \in X_{n+1}\setminus X_n=D_{n+1}$.

  Conversely, suppose $\alpha_i(D_n) \subset D_{n+1}$ for all $0\le i \le n$. Consider  $x \in X^{\alpha_n}\cap X_n$. Thus $x \in X^{\alpha_n}\cap D_k$ for some $0\le k \le n$.
Suppose $x \in X^{\alpha_n}\cap D_n$ then $x=\alpha_n(x) \in D_{n+1}$ (by our assumption), which contradicts $D_n\cap D_{n+1} = \emptyset$. Therefore $x \in D_k$ for some $0 \le k < n$, i.e. $x \in X_{n-1}$. It follows that $X_{n-1}=X_\infty^{\alpha_n}\cap X_n$.
\end{proof}

\begin{rems}
  \begin{enumerate}
    \item While Theorem \ref{T:scs-saturation} follows partially from Theorem \ref{T:saturation}, we provided an elementary proof to emphasise that von Neumann's mean ergodic theorem is not required in the setting of SCSs.
  \item As a consequence of Theorem \ref{T:scs-saturation} there is no
    ambiguity in using the analogue for saturation at level $n-1$ up
    to level $n$ for an SCS $(X_k,\delta_i)_{k,i}$ or its induced SCH
    $(\CH_k,\delta_i)_{k,i}$ since for all $n\in\NN_0$ we have that
    $X_{n-1}=X_\infty^{\alpha_n}\cap X_n$ if and only if
    $\CH_{n-1}=\CH_\infty^{\alpha_n}\cap\CH_n$.
\end{enumerate}

\end{rems}
\section{Cohomology}\label{S:cohomology}

An SCH $(\CH_k, \delta_i)_{k,i}$ yields a cochain complex $(\CH^n,\diff^n)$ with $\CH^n:=\CH_n$ for all $n\ge -1$ and
\[
\begin{tikzcd}
& \CH^{-1} \arrow[r, "\diff^{-1}"] & \CH^0 \arrow[r, "\diff^0"] & \CH^1 \arrow[r, "\diff^1"] & \ldots
\end{tikzcd}
\]
with coboundary map

\[
\diff^{n} = \sum^{n+1}_{i=0} (-1)^{n+1-i} \delta_i = (-1)^{n+1} \sum^{n+1}_{i=0} (-1)^i \delta_i \quad\quad (n \ge -1).
\]

We refer to $x \in \CH^n$ as an $n$-cochain (for $n \ge -1$). For the discussion of cohomology in this section we adopt standard conventions (such as upper indices), as for example in \cite{MacLane1995}. 
We include a (not always used) factor $(-1)^{n+1}$ in the definition of $\diff^{n}$. This convention simplifies some of the following formulas and in fact it is used in \cite{MacLane1995}, Chapter II.3, as well. Clearly such factors do not change the cohomology. It is also convenient to put $\CH^{-2} := \{0\}$ and $\diff^{-2} := 0$.

As explained in the previous section, for an SCH we interpret $\delta_{n+1}$ as an embedding of $\CH^{n}$ into
$\CH^{n+1}$, so an $n$-cochain can also be interpreted as an $(n+1)$-cochain. Let us discuss some features arising from that. First note that the extension $\sum^{n+1}_{i=0} (-1)^{n+1-i} \alpha_i$, which we will denote by the same symbol $\diff^n$, makes sense as a (bounded linear) operator on $\CH_\infty$ and will be used in this way whenever convenient. Because $\CH^k \subset \CH^{\alpha_{k+1}}$ we then have (for all $k \ge -1$)
\begin{align*}
x \in \CH^k \; \text{is a $k$-cocycle} 
& \quad \Leftrightarrow \quad \diff^{k}(x) = 0 \\
& \quad \Leftrightarrow \quad x = \diff^{k-1}(x) 	
\end{align*}
(with $\diff^{-2} = 0$).
So we can think of $k$-cocycles as fixed points of $\diff^{k-1}$ acting on $\CH^k$. Further
\begin{align*}
x \in \CH^k \; \text{is a $k$-coboundary} 
& \quad \Leftrightarrow \quad x = \diff^{k-1}(y) \; \text{with} \; y \in \CH^{k-1}	
\end{align*}
(with $\CH^{-2} = \{0\}$). 

Let us write $\CH_c^k$ for the space of all $k$-cocycles and $\CH_b^k$ for the space of all $k$-coboundaries. We have $\CH_b^k \subset \CH_c^k \subset \CH^k$ because we have a cochain complex. 

\begin{proposition} \label{prop:d}
$\CH_c^k \cap \CH^{k-1} = \CH_b^k \cap \CH^{k-1}$
\end{proposition}
\begin{proof}
If $x \in \CH_c^k \cap \CH^{k-1}$ then $x= \diff^{k-1}(x)$ with $x \in \CH^{k-1}$.	
\end{proof}	

\begin{lemma} \label{lem:d}
If $x \in \CH^\ell$ with $-1 \le \ell \le k$ then
\[
\diff^{k}(x) = 
\begin{cases}
  x - \diff^{\ell-1}(x) & \text{if } k-\ell \text{ is even} \\
  \diff^{\ell-1}(x) & \text{if } k-\ell \text{ is odd}.
\end{cases}
\]
\end{lemma}

\begin{proof}
	$\alpha_i (x) = x$ if $i>\ell$.
	\end{proof}

\begin{proposition}
	$\CH_c^{k+2} \cap \CH^{k} = \CH_c^{k} \quad(\subset \CH_b^{k+2})$.
\end{proposition}

\begin{proof}
The equality follows from Lemma \ref{lem:d}, in fact if $x \in \CH^{k}$ then $\diff^{k+2}(x) = x - \diff^{k-1}(x)$ and so $\diff^{k+2}(x)=0$ if and only if $x = \diff^{k-1}(x)$. Use Proposition \ref{prop:d} to get the inclusion.	
\end{proof}

We see that we can build the space of all even cocycles (i.e, $k$-cocycles for some even $k$) step by step, starting with small values of $k$ and then adding additional $k$-cocycles if $k$ is increased. The same applies for odd cocycles. 

In the case of a saturated SCH we can get a rather complete picture.

\begin{theorem}\label{T:cohomology}
For an SCH that satisfies $\alpha_k(\CD_k) \subset \CD_{k+1}$ for some $k \in \Nset_{0}$ we have $\CH_c^k = \CH_b^k$, i.e., the $k^\text{th}$ cohomology group is trivial. 
Explicitly we get
\[
\CH_c^k  = \CH_b^k 
= \diff^{k-1} \Big(  \bigoplus_{\substack{\ell:\; -1 \le \ell < k \\ 
                      \text{ with } k - \ell \text{ odd}}} 
                      \CD_\ell \Big) 
=  \bigoplus_{\substack{\ell:\; -1 \le \ell < k \\ 
                      \text{ with } k - \ell \text{ odd}}}  
                      (\id - \diff^{\ell-1})\; \CD_\ell
\]
\end{theorem}

\begin{proof}
Fix $k \in \Nset_0$ and let $x \in \CH_c^k$. Then $x=y+z$ with $y \in \CH^{k-1}$ and $z \in \CD_k$ and
\[
x = \diff^{k-1}(x) = \diff^{k-1}(y) + \diff^{k-1}(z).
\] 
We have $\diff^{k-1}(y) \in \CH_b^k \subset \CH^k$.
Also $\diff^{k-1}(z) = x - \diff^{k-1}(y) \in \CH^k$.
 As $\alpha_k(\CD_k)\subset \CD_{k+1}$ we have, by Lemma \ref{L:one4all}, $\alpha_i (\CD_k) \subset \CD_{k+1}$ for all $0 \le i \le k$, thus $\diff^{k-1}(\CD_k) \subset \CD_{k+1}$. It follows that $\diff^{k-1}(z) \in \CH^k \cap \CD_{k+1} = \{0\}$, hence $\diff^{k-1}(z)=0$ and $x = \diff^{k-1}(y) \in \CH_b^k$. We have proved that
 $\CH_c^k = \CH_b^k$.
 
 Combining $\diff^{\ell-1}(\CD_\ell) \subset \CD_{\ell+1}$ with Lemma \ref{lem:d} allows us to compute these cocycles explicitly. In fact, writing $x \in \CH_c^k$ as 
 \[
 x = (x_{-1}, x_0, \ldots, x_k) \in \bigoplus^k_{\ell= -1} \CD_\ell
 \]
 we find the component $\diff^{k}(x)_\ell \in \CD_\ell$ of $\diff^{k}(x)$ to be (with $-1 \le \ell \le k+1$)
 \[
 \diff^{k}(x)_\ell = 
 \begin{cases}
    \diff^{\ell-2}(x_{\ell-1}) + x_\ell & \text{if } k-\ell \text{ is even} \\
 - \diff^{\ell-2}(x_{\ell-1}) & \text{if } k-\ell \text{ is odd}.
  \end{cases}
 \]
 We use the convention that everything is equal to $0$ which contains a super- or subscript less than $-1$.
 The formulas follow by applying Lemma \ref{lem:d} to $x_{\ell} \in \CD_{\ell} \subset \CH^\ell$ and making use of the fact that by Theorem \ref{T:saturation} 
we have $\diff^{\ell-2}(x_{\ell-1}) \in \CD_\ell$.
 
 We have $x \in \CH^k_c$ if and only if $\diff^{k}(x)=0$. To obtain $\diff^{k}(x)_\ell =0$ for $k-\ell$ even we need $x_\ell = -\diff^{\ell-2}(x_{\ell-1})$ (and if $k-(-1)=k+1$ is even then $x_{-1}=0$). To obtain further that $\diff^{k}(x)_\ell =0$ for $k-\ell$ odd we can choose the $x_\ell \in \CD_\ell$ independently and arbitrarily. In fact, in this case $k-(\ell-1)$ is even and if we have $x_{\ell-1} = \diff^{\ell-3}(x_{\ell-2})$ (as required in the even case) then 
 \[
 \diff^{k}(x)_\ell = 
 -\diff^{\ell-2}(x_{\ell-1}) = \diff^{\ell-2}d^{\ell-3}(x_{\ell-2}) = 0
 \]
automatically. We conclude that $x \in \CH_c^k$ if and only if it is of the form
\begin{align*}
x & = (\ldots,x_{k-3}, -\diff^{k-4}(x_{k-3}), x_{k-1}, -\diff^{k-2}(x_{k-1})) \in \bigoplus^k_{\ell= -1} \CD_\ell \\
& = \; \diff^{k-1} (\ldots, 0, x_{k-3},0,x_{k-1})
\end{align*}
with $x_{-1}=0$ if $k-(-1)=k+1$ is even and with independent and arbitrary choices of $x_\ell \in \CD_\ell$ if $k-\ell$ is odd. This proves the theorem.
\end{proof}	
	
\begin{corollary}\label{sat-c}
  An SCH that is saturated at level $k-1$ has trivial $k^\text{th}$ cohomology group.  Hence, a saturated SCH has trivial cohomology.
\end{corollary}

\begin{proof}
  Follows immediately from Theorem \ref{T:saturation} and Theorem \ref{T:cohomology}.
\end{proof}

We will not study the cohomology of non-saturated SCHs systematically in this paper but we can illustrate a few phenomena by using examples of SCSs of the kind already seen in Examples \ref{E:basic_examples}(2).
So we consider SCHs from SCSs built with the prototypical partial shifts on $\NN_0$. We can write
$x \in \CH^n$ in the form $x = x_0 0 + x_1 1 + \ldots + x_n n$. Here $0, \ldots, n$ have to interpreted as vectors in an orthonormal basis and the $x_i$ are complex coefficients. In the non-saturated cases some of these basis vectors are actually missing (and the $\CH^n$ have a lower dimension), we can take this into account by just choosing the corresponding complex coefficients to be zero. 

With these conventions the coboundary map $\diff^{-1}$ is zero, and for $n$ even
\begin{align*}
& \diff^{n-1}: \; \CH^{n-1} \rightarrow \CH^{n}, \\
& x_0 0 + x_1 1 + \ldots + x_{n-1} (n-1) \mapsto (x_0+x_1) 1 + (x_2+x_3) 3 + \ldots + (x_{n-2}+x_{n-1}) (n-1),
\end{align*}
\begin{align*}
& \diff^{n}: \; \CH^{n} \rightarrow \CH^{n+1}, \\
& x_0 0 + x_1 1 + \ldots + x_{n} (n) \mapsto x_0 0 - x_0 1 + x_2 2 - x_2 3 + \ldots + x_n n - x_n (n+1)
\end{align*}

\begin{examples}
  \begin{enumerate}
  \item $\ell(0)=1$, $\ell(n)=n$ for $n \ge 1$. In this case $\diff^{-1}$ and $\diff^0$ are zero and 
    $\diff^1 (x_0 0 + x_1 1) = (x_0+x_1) 1$, hence $ker\, \diff^1$ is spanned by $0-1$. So the cohomology of level $1$
    is one-dimensional. On other levels the cohomology is trivial. 

   \item $\ell(0)=2$, $\ell(n)=n$ for $n \ge 1$. Again $\diff^{-1}$ and $\diff^0$ are zero but now $0 \notin \CH^1$ and
    $\diff^1(x_1 1) = x_1 1$, so $ker\, \diff^1$ is trivial and $im\, \diff^1$ is spanned by $1$. Further 
    $\diff^2(x_0 0 + x_1 1 + x_2 2) = x_0 0 - x_0 1 + x_2 2 - x_2 3$, so $ker\, \diff^2$ is also spanned by $1$. Higher levels are unchanged from the saturated case. So we see that in this example the cohomology is trivial on all levels. The example shows that a non-saturated SCH can also have trivial cohomology, there is no immediate converse to the statement in Corollary \ref{sat-c}.
  \end{enumerate}
\end{examples}

\section{Hessenberg Factorizations and Braided SCHs}

We discuss another setting which will be useful to describe partial shifts, in particular by establishing certain connections to braid group representations. 
Let
\[
\CH_{-1} \subset \CH_0 \subset \ldots \subset \CH_\infty := \overline{\bigcup_{-1 \le k \in \Zset}{\CH_k}}
\]
be a tower of Hilbert spaces and
suppose that a sequence of unitaries $(u_m)_{m \ge 1}$ on $\CH_\infty$ satisfies
\begin{align*}
(H1) & \quad \mbox{for all } k \ge 1, \quad \CH_{k-2} \mbox{ is in the fixed point space of } u_{k}\\
(H2) & \quad \mbox{for all } \ell \ge k \ge 1, \quad \CH_{\ell} \mbox{ is an invariant subspace for } u_{k}\\
(C) & \quad u_m \mbox{ commutes with } u_n \mbox{ whenever } |m-n| \ge 2.
\end{align*}
Because of (H1) we can define a sequence of isometries
$(\alpha_n)_{n \in \Nset_0}$ on $\CH_\infty$ by
\[
\alpha_n := sot-\lim_{m \to \infty} u_{n+1} u_{n+2} \ldots u_m,
\]
where sot refers to the strong operator topology. In fact,
\[
\alpha_n |_{\CH_k} = 
\left\{
\begin{array}{ll}
id & \mbox{if } k<n \\
u_{n+1} \ldots u_{k+1} & \mbox{if } k \ge n
\end{array}
\right.
\]
Note that
$\alpha_{n-1} = u_n \alpha_n$ for all $n \ge 1$
and $u_m$ commutes with $\alpha_n$ whenever
$1 \le m < n$. Because of (H2),
for all $n \ge 0$ and $k \ge -1$, we have $u_{k+1} \CH_k \subset u_{k+1} \CH_{k+1} = \CH_{k+1}$ and
$\alpha_n \CH_k \subset \CH_{k+1}$, in other words: adaptedness is automatic. Note that this means that the block matrices for the $\alpha_n$ (with respect to innovation spaces $\CD_k := \CH_k \ominus \CH_{k-1}$) have Hessenberg form: upper triangular except for entries just below the diagonal. This motivates the following terminology. 

\begin{definition}
We refer to $(\CH_k, u_m)_{k,m}$ satisfying (H1), (H2), (C) as a Hessenberg factorization for the (adapted) isometries $(\alpha_n)_{n \in \Nset_0}$ on $\CH_\infty$.
\end{definition}

Hessenberg factorizations appear naturally.

Suppose $\alpha_0: \CH \rightarrow \CH$ is an isometry. We define $\CH_{-1}$ to be the subspace of fixed points of $\alpha_0$ and then $(\CH_k)$ any tower so that $\alpha_0(\CH_{k-1}) \subset \CH_k$ for all $k$. If inductively for each $k \ge 1$ we can find a unitary $u_k$ on $\CH_k$, extended by the identity to the orthogonal complement of $\CH_k$, so that $u_1 \ldots u_k$ coincides with $\alpha_0$ on $\CH_{k-1}$ (for example this is always possible if the $\CH_k$ are finite dimensional) then we obtain a Hessenberg factorization for $(\alpha_n)$ as above. What is achieved is that we write the isometry $\alpha_0$ as an infinite product of unitaries which are well localized with respect to the tower and which are in this sense simpler than the original isometry. For example in \cite{FF90}, XV.2 this is done for the minimal isometric dilation of a positive definite sequence and in \cite{Go04}, 3.1.2 there is a discussion of more general cases.

On the other hand suppose that we have a sequence of unitaries $(u_m)$ satisfying (C). Then we can define 
\[
\hat{\CH}_{k} := \CH^{u_\ell, \ell \ge k+2} := \{ x \in \CH_\infty \colon \quad u_\ell\, x = x \quad \mbox{for all } \ell \ge k+2 \}
\]
(for all $k \ge -1$) and obtain a Hessenberg factorization for a sequence of isometries as above. We call $(\hat{\CH}_k, u_m)_{k,m}$
saturated.

Note that if we start with any Hessenberg factorization
$(\CH_k, u_m)_{k,m}$ then
$\CH_{k} \subset \hat{\CH}_{k}$ for all $k$
and $(\hat{\CH}_k, u_m)_{k,m}$ provides another (saturated) Hessenberg factorization for the same isometries which we call the saturation.



\vspace{0.2cm}
We are interested in Hessenberg factorizations for partial shifts. 

\begin{lemma}\label{lem:braided}
	Suppose we have a Hessenberg factorization $(\CH_k, u_m)_{k,m}$ for the
	 isometries $(\alpha_n)_{n \in \Nset_0}$ on $\CH_\infty$. 
The following are equivalent: 
\begin{itemize}
	\item[(1)] 
	$u_{j+1} \alpha_i = \alpha_i u_j$ for all $i < j$ 
    \item[(2)] 
    $u_{j+1} \alpha_{j-1} = \alpha_{j-1} u_j$ for all $j \ge 1$ 
    \item[(3)] 
    $u_j u_{j+1} u_j = u_{j+1} u_j u_{j+1}$ on $\alpha_{j+1}(\CH_\infty)$, for all $j \ge 1$	
\end{itemize}	
If one (and hence all) of the conditions $(1) - (3)$ is valid then $(\alpha_n)_{n \ge 0}$
is a sequence of partial shifts.	
	 	\end{lemma}

\begin{proof}
$(1)\Rightarrow(2)$ is obvious. The converse $(2)\Rightarrow(1)$ follows (assuming (2) and $i < j-1$) from
\begin{align*}
u_{j+1} \alpha_i 
& = u_{j+1} u_{i+1} \ldots u_{j-1} \alpha_{j-1}
= u_{i+1} \ldots u_{j-1} u_{j+1} \alpha_{j-1} \\
& = u_{i+1} \ldots u_{j-1} \alpha_{j-1} u_j
= \alpha_i u_j
\end{align*}

$(2)\Leftrightarrow(3)$ follows from
\begin{align*}
u_{j+1} \alpha_{j-1} & = u_{j+1} u_j u_{j+1} \,\alpha_{j+1} \\
\alpha_{j-1} u_j = u_j u_{j+1}\,\alpha_{j+1} u_j & = u_j u_{j+1} u_j \,\alpha_{j+1}
\end{align*}

To show the remaining part we use (1). Let $i<j$ and $x \in \CH_k$ for some $k$. Then we have $\alpha_i(x) \in H_{k+1}$ and 
\[
\alpha_j \alpha_i x = u_{j+1} \ldots u_{k+2} \alpha_i x = \alpha_i u_j \ldots u_{k+1} x = \alpha_i \alpha_{j-1} x
\]	
For $x \in \CH_\infty$ the result follows by approximation. 	
	\end{proof}

Suppose now that we are given a unitary representation of the infinite braid group $\Bset_\infty$
on a Hilbert space $\CH$. 
For simplicity of notation we denote both the abstract Artin generators and the unitaries representing them by $(\sigma_m)_{m \in \Nset}$. 
We have the braid relations
\begin{align*}
\text{(B1)} & \quad 
\sigma_m \sigma_{m+1} \sigma_m = \sigma_{m+1} \sigma_m \sigma_{m+1} \mbox{ for all } m \ge 1;\\
\text{(B2)} & \quad \sigma_m \mbox{ commutes with } \sigma_n \mbox{ whenever } |m-n| \ge 2.
\end{align*}

If for $k \ge -1$ we define
\[
\hat{\CH}_k := 
\CH^{\sigma_\ell, \ell \ge k+2}
:= \{ x \in \CH \colon \quad \sigma_\ell\, x = x \quad \mbox{for all } \ell \ge k+2 \}
\]
we obtain a tower
\[
\hat{\CH}_{-1} \subset \hat{\CH}_0 \subset \hat{\CH}_1 \subset \ldots \subset \hat{\CH}_\infty := \overline{\bigcup_{1 \le k \in \Zset}{\hat{\CH}_k}}.
\]
Only the subrepresentation on $\hat{\CH}_\infty$ is relevant for the following constructions. 

\begin{theorem}\label{thm:canon}
Given a unitary representation of $\Bset_\infty$ on $\CH$ (as above). Then $(\hat{\CH}_k, \sigma_m)_{k,m}$ provides a saturated Hessenberg factorization for a sequence 
$(\alpha_n)_{n \in \Nset_0}$
of partial shifts.	
\end{theorem}

\begin{proof}
From (B2) we have (C) for the sequence $(\sigma_m)$ of unitaries, so our choice of tower provides the corresponding saturated Hessenberg factorization for the $\alpha_n$ defined by 
$\alpha_n = \operatorname{sot}-\lim_{m \to \infty} \sigma_{n+1} \sigma_{n+2} \ldots \sigma_m$. Now (B1) together with Lemma \ref{lem:braided} yields that the $\alpha_n$ are partial shifts.

\end{proof}

\begin{definition}
The SCH $(\hat{\CH}_k, \delta_i)_{k,i}$ on $\hat{\CH}_\infty$ which is associated to the partial shifts given in Theorem \ref{thm:canon} is called the canonical SCH of the unitary representation of $\Bset_\infty$. 	
\end{definition}

\begin{proposition}\label{prop:braided}
The canonical SCH is saturated, i.e., for all $n \in \Nset_0$

\[
\hat{\CH}_\infty^{\alpha_n} = \hat{\CH}_{n-1} \;(:= \CH^{\sigma_\ell, \ell \ge n+1}).
\]

\end{proposition}

\begin{proof}
What we need to check here is that our definitions of saturatedness for an SCH and for a Hessenberg factorization are consistent for the canonical SCH of a braid group representation. 
It is clear that $\hat{\CH}_{n-1} = \CH^{\sigma_\ell, \ell \ge n+1} \subset \hat{\CH}_\infty^{\alpha_n}$. To prove the inclusion $\hat{\CH}_\infty^{\alpha_n} \subset \CH^{\sigma_\ell, \ell \ge n+1}$
assume that $\alpha_n(x)=x$. Hence also
$\alpha_{n+1}(x)=x$ and we get
\[
\sigma_{n+1}(x) = \sigma_{n+1} \alpha_{n+1}(x) = \alpha_n(x) = x.
\]
A similar argument applies for all $\sigma_\ell$ with $\ell \ge n+1$.
\end{proof}

\begin{proposition} \label{prop:fix}
For an SCH $(\CH_k, \delta_i)_{k,i}$ the following two properties are equivalent:
\begin{itemize}
    \item[(a)] 
    There exists a unitary representation of $\Bset_\infty$ on $\CH_\infty$ so that $\delta_i: \CH_{n-1} \rightarrow \CH_n$
    can be factorized as
    \[
    \delta_i = \sigma_{i+1} \ldots \sigma_{n+1} |_{\CH_{n-1}} \quad i=0,\ldots n,\;\; n \in \Nset_0.
    \]
    \item[(b)]
    The saturation of the SCH is the canonical SCH of a unitary representation of $\Bset_\infty$.
\end{itemize}
\end{proposition}

\begin{proof}
Given (b) note that the space $\CH_\infty$ does not change when we go to the saturation. Now using the representation of $\Bset_\infty$ for which this saturation is canonical we see that by restricting to $\CH_{n-1} \subset \hat{\CH}_{n-1} = 
\CH^{\sigma_\ell, \ell \ge n+1}$ we get the factorizations of the $\delta_i$ stated in (a).

Conversely, if the representation in (a) is given we can consider its canonical SCH which lives on $\CH_\infty$.
If $\ell \ge n+1$ and $x \in \CH_{n-1}\;$ ($\subset \CH_{\ell-2}$) then, because $\delta_{\ell-1}$ on $\CH_{\ell-2}$ acts identically, we have $\sigma_\ell x = \delta_{\ell-1} x = x$.
Hence $\CH_{n-1} \subset \hat{\CH}_{n-1} = 
\CH^{\sigma_\ell, \ell \ge n+1}$.
The factorizations for the $\delta_i$ ensure further that the partial shifts for the original SCH and for this canonical SCH are the same. Now it follows from Proposition \ref{prop:braided} that the saturation of the original SCH is exactly this canonical SCH.
\end{proof}

\begin{definition} \label{def:braided}
	An SCH is called braided if it satisfies the properties in Proposition \ref{prop:fix}. 	
\end{definition}

\begin{corollary}\label{cor:braided}
An SCH is the canonical SCH of a unitary representation of $\Bset_\infty$ if and only if it is braided and saturated.
\end{corollary}


Note that Proposition \ref{prop:braided} insures that saturation as an SCH and the saturation of the Hessenberg factorization gives the same result for braided SCH.

Finally consider the infinite symmetric group $\Sset_\infty$. It is a quotient of $\Bset_\infty$, hence we can think of any representation of $\Sset_\infty$ as a representation of $\Bset_\infty$. This leads to the following natural extension of our terminology.

\begin{definition} \label{def:symmetric}
	The SCH $(\hat{\CH}_k, \delta_i)_{k,i}$ on $\hat{\CH}_\infty$ which is associated to the partial shifts given in Theorem \ref{thm:canon} but with a representation of $\Sset_\infty$ is called the canonical SCH of this representation of $\Sset_\infty$.
	An SCH is called symmetric if its saturation is the canonical SCH of a representation of $\Sset_\infty$. 	
\end{definition}

\section{The Graph of $\ssc$}

We may analyse the finer structure of an SCH by considering the relationships between the images of the morphisms of the semi-simplicial category $\ssc$. This we do in this section by constructing a graph associated with $\ssc$. Not only will this provide a co-ordinate system for SCHs, which we study in detail in Section \ref{S:labeled_subspaces}, but also a means to produce examples.

In its most basic form the graph can be thought of as a directed graph with finite subsets of $\NN_0$ as vertices and edges from vertex $u$ to $\e_i(u)$ for $0\le i \le \max u$ (where $\e_i$ is defined in Section \ref{S:intro}, or see below for a more detailed treatment). This directed graph is drawn in Figure \ref{F:Lambda_S} (where the colours can be ignored for now). Note that the choice of $\e_i$ is not unique, e.g. $\e_1(\{0,2\})=\e_2(\{0,2\})=\{0,3\}$. However, we will make use of some additional structure in the sequel, which we describe below.


First we note some basic facts about $\ssc$ and SCHs, and establish some notation.  
%
Each morphism $f$ of $\ssc$ is uniquely determined by its image $\im f$ and codomain $[n]$ (its domain $[m]$ is given by $m=\card{\im f}-1$). Moreover, $f:[m] \into [n]$ ($m<n$), can be written uniquely as

\[
f=\e_{i_k}^{[n]}\e_{i_{k-1}}^{[n-1]}\cdots \e_{i_1}^{[m+1]}, \text{ with } 0 \le i_1< i_2 <\cdots < i_k\le n,
\]
where $\im f = [n]\setminus \{i_1,i_2,\ldots,i_k\}$, $k:=n-m$, and $\e_p^{[l]}:[l-1] \into [l]$ is given by $\im \e_p^{[l]}=[l]\setminus \{p\}$ for all $l\in \NN_0,\; 0\le p \le l $.

  We write $\delta_f:=\Ff(f)$ and, for ease of notation, we write $\e_i$ for $\e_i^{[l]}$ when information about the domains and codomains
  is unnecessary or clear from the context, and $\delta_i$ for $\delta_{\e_i}$.    Thus, it is in this sense that an SCH is given by $\Hh_n = F[n]$ and $\delta_i=F(\e_i)$, as in Section \ref{sec:defs}.

We identify the identity morphisms of a (small) category $\cat$ with their corresponding objects and let $\cat^0$ denote the set of identity morphisms. Thus with little risk of confusion we denote the set of morphisms of $\cat$ by $\cat$ also, and write $\lambda \in \cat$ to mean that $\lambda$ is a morphism in $\cat$. We write the source (domain) and range (codomain) of a morphism $\lambda$ as $s(\lambda)$ and $r(\lambda)$, which give maps $s:\Lambda \into \Lambda^0$ and $r:\Lambda \into \Lambda^0$, respectively. For objects $u,v$ we let $v\cat:=r^{-1}(v)$, $\cat u:=s^{-1}(u)$ and $v \cat u := v\cat \cap \cat u$.


Consider the category $\Lambda_S$ with objects $\Lambda_S^0:=\{u  \subset \NN_0 : \text{ $u$ is finite}\}$ and, for
  $u,v \in \Lambda_S^0$, morphisms
  \begin{align*}
    v\Lambda_Su &= \{ \lambda \in v^u : \im \lambda = v,\; \exists\;  f \in \ssc \text{ such that } u \subseteq s(f),\; \lambda(x)=f(x)\; \forall\; x \in u \}\\
                &= \{\lambda \in v^u : \im\lambda = v,\; \forall\; x,y\in u,\; x<y \implies 0 \le \lambda(x)-x \le\lambda(y)-y \}
  \end{align*}
  Composition of morphisms is the composition of mappings.  Note that $\emptyset \in \Lambda_S^0$ and $\emptyset\Lambda_S\emptyset = \{\emptyset\}$.

Thus morphisms of $\Lambda_S$ are strictly increasing bijections and $\card{v\Lambda_Su} \le 1$, with $\card{u\Lambda_Sv}=\card{v\Lambda_Su}=1$ implying $u=v$ for all $u,v \in \Lambda_S^0$. Therefore $\Lambda_S$ is a partial order (in the sense of Mac Lane \cite[p.11]{MacLane1971}); it is isomorphic to the thin category associated to the partial order $\le$ on $\{u \subset \NN_0 : \text{$u$ is finite}\}$ given by
\[ u \le v \iff   v = f(u) \text{ for some $f \in \ssc$ with } u \subseteq s(f). \]
Hence we may also identify a morphism $\lambda \in \Lambda_S$ with $(r(\lambda), s(\lambda))$ and thus identify $\Lambda_S$ with
  \begin{align*}
    & \{(f(u),u) \in \Lambda_S^0 \times \Lambda_S^0 : f \in \ssc,\; u \subseteq s(f) \} \\
    = & \{(f(u),u) \in \Lambda_S^0 \times \Lambda_S^0 :  \text{$f$ is a strictly increasing map $f$ on $\NN_0$} \}.
  \end{align*}

 We may therefore form a directed graph associated with $\Lambda_S$ by taking $\Lambda_S^0$ as the vertex set and forming an edge for each $(v,u) \in \Lambda_S$ such that, for all $w \in \Lambda^0_S$, $u \le w \le v$ implies $w\in \{u,v\}$.  The edge set is therefore $\{(\e_i(u),u) : u \in \Lambda_S^0, 0\le i\le \max{u}\}$. The directed graph is illustrated in Figure \ref{F:Lambda_S} (if one ignores the colours).
  
  \begin{rem}
Not all strictly increasing bijections between finite subsets of $\NN_0$ are morphisms in $\Lambda_S$. For example, the strictly increasing bijection $\{0,2\} \into \{1,2\}, 0 \mapsto 1, 2\mapsto 2$ is not in $\Lambda_S$.
    \end{rem}
 
  Yet another description of $\Lambda_S$, which we will use in Section \ref{S:SCS_ext}  is as the partial order determined by the binary relation $\le$ on $\bigsqcup_{k=0}^\infty\NN_0^k$ (where $\NN^0_0:=\{\emptyset\}$) given by applying the usual order on $\NN_0$ coordinatewise to $\NN_0^k$, and $\emptyset \le \emptyset$. Explicitly, the latter partial order (viewed as a category) is $\mho_\infty:=\bigsqcup_{k\in\NN_0}\mho_k$ where for $k\in\NN_0$, $\mho_k:=\{ (n,m) : m,n \in \NN^k_0, m \le n\},\; r((n,m))=(n,n)\equiv n,\; s((n,m))=(m,m)\equiv m,\; (n,m)(m,l)=(n,l)$. An isomorphism $\Lambda_S \into \mho_\infty$ is given by
the bijection from $\Lambda_S^0$ onto $\mho_\infty^0$ defined by:
  \[ v=\{v_1,v_2,\ldots,v_k\} \mapsto (v_1,v_2-v_1-1,v_3-v_2-1,\ldots,v_k-v_{k-1}-1),\; \emptyset \mapsto \emptyset\]
  where $v_i<v_{i+1}$ for all $i$.

Now, for each $k \in \NN$, $\mho_k$ equipped with the degree functor $d_k:\mho_k\into \NN_0^k$ given by $d_k(n,m)=n-m$, is a rank $k$ graph \cite{KumjianPask2000}.  Moreover, $\mho_0$ with degree functor $d_0:\mho_0 \into \{0\}$ can be viewed as a rank 0 graph.  Extending to $\mho_\infty$  (and identifying a finite $k$-tuple with its canonical embedding in $\NN_0^\infty$) equips $\mho_\infty$ with a degree functor $d:\mho_\infty \into \NN_0^\infty$ satisfying the required factorisation property for $(\mho_\infty,d)$ to be an $\NN_0^\infty$-graph in the sense of \cite{BrownloweSimsVittadello2013}.
Thus we may view $\Lambda_S$ as an $\NN_0^\infty$-graph by pulling the degree functor back from $\mho_\infty$, so that for example $d(\varepsilon_i(u),u)=e_j$ if $u_j \le i < u_{j+1}$, where $u_n$ is the $n^\text{th}$ element of $u \in \Lambda^0_S$ in ascending order and $e_j$ is a canoncial generator of the monoid $\NN_0^k$. For all $n \in \NN_0^\infty$, we set $\Lambda_S^n:=d^{-1}(n)$, noting that by the factorisation property $d^{-1}(0)$ is indeed the previously defined $\Lambda_S^0$, which we view as the set of vertices of the graph.  With this in mind we make the following definitions.

\begin{defn}\label{D:rank-lev} By the graph of $\Delta_S$ we mean the $\NN_0^\infty$-graph $\Lambda_S$ constructed above.
The rank of a morphism $\lambda \in \Lambda_S$, denoted $\rank{\lambda}$, is defined to be $\card{s(\lambda)}$. The level of a vertex $v\in \Lambda_S^0$, denoted $\lev{v}$, is defined to be
  \[
    \lev{v}:=
    \begin{cases}
      \max v & \text{if } v \neq \emptyset,\\
      -1 & \text{if } v = \emptyset.
    \end{cases}
  \]
\end{defn}

We have $\Lambda_S = \bigsqcup_{k \in \NN_0} \Lambda_k$ where $\Lambda_k:=\{\lambda \in \Lambda_S : \rank{\lambda}=k\}\cong \mho_k$ is the sub-category of all morphisms of rank $k$, which when equipped with the restriction of the degree functor, is a $k$-graph.  Thus $\Lambda_S$ decomposes into disjoint connected $k$-graphs $\Lambda_k$ (in the sense that $u\Lambda_kv\cup v\Lambda_ku \neq \emptyset$ for every $u,v \in \Lambda^0_k$).

We may also visualise $\Lambda_S$ by means of its skeleton \cite[Definition 4.1]{HazlewoodRaeburnSimsWebster2013} (or more precisely, an obvious extension of the concept of the skeleton of a $k$-graph) as in Figure \ref{F:Lambda_S}, where the edges are morphisms of minimal non-zero degree with each minimal non-zero degree assigned a colour.

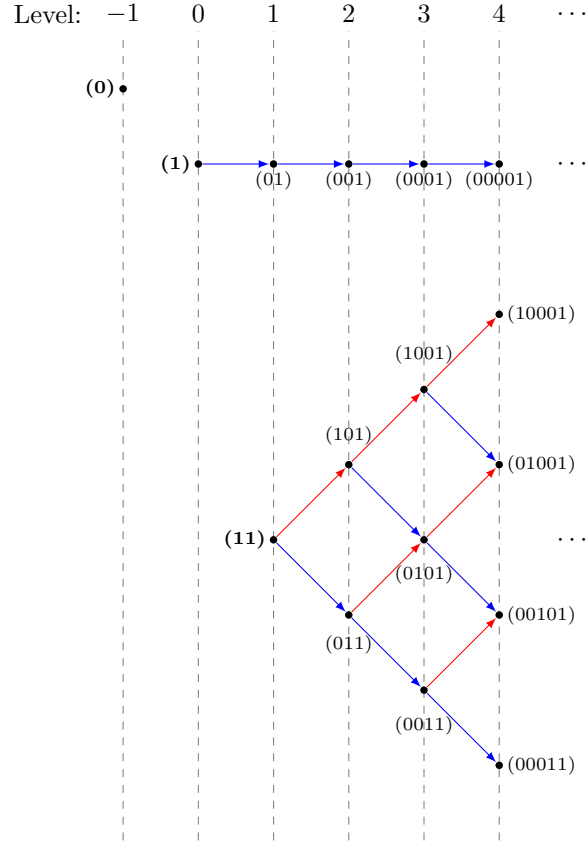
\begin{figure}[!h]
  \begin{tikzpicture}
    \node at (-1,1) {Level:};
 \node (D-1) at (0,1) {$-1$};
 \node (D0) at (1,1) {$0$};
 \node (D1) at (2,1) {$1$};
 \node (D2) at (3,1) {$2$};
 \node (D3) at (4,1) {$3$};
 \node (D4) at (5,1) {$4$};
 \node at (6,1) {$\dots$};
 
 \pgfmathtruncatemacro{\b}{-10}
 
  \draw [dashed,gray] (0,\b)-- (D-1);
  \draw [dashed,gray] (1,\b) -- (D0);
  \draw [dashed,gray] (2,\b) -- (D1);
  \draw [dashed,gray] (3,\b) -- (D2);
  \draw [dashed,gray] (4,\b) -- (D3);
  \draw [dashed,gray] (5,\b) -- (D4);
 
  \node [vertex] (0) at (0,0) {};
  \node [vertex] (1) at (1,-1) {};
  \node [vertex] (01) at (2,-1) {};
  \node [vertex] (001) at (3,-1) {};
  \node [vertex] (0001) at (4,-1) {};
  \node [vertex] (00001) at (5,-1) {};

  \node [vertex] (11) at (2,-6) {};
  \node [vertex] (101) at (3,-5) {};
  \node [vertex] (011) at (3,-7) {};
  \node [vertex] (1001) at (4,-4) {};
  \node [vertex] (0101) at (4,-6) {};
  \node [vertex] (0011) at (4,-8) {};
  \node [vertex] (10001) at (5,-3) {};
  \node [vertex] (01001) at (5,-5) {};
  \node [vertex] (00101) at (5,-7) {};
  \node [vertex] (00011) at (5,-9) {};

  \node at (6,-1) {$\cdots$};
  \node at (6,-6) {$\cdots$};
  \tikzset{every node/.style={above=2pt, midway,font=\scriptsize,text=black}}
  
  \draw [dedge] (1)--  node {$\e_0$} (01);
  \draw [dedge] (01)-- node {$\e_1$} (001);
  \draw [dedge] (001)-- node{$\e_2$} (0001);
  \draw [dedge] (0001)-- node{$\e_3$} (00001);

  \draw [bdedge] (11)-- node{$\e_0$} (011);
  \draw [bdedge] (011)-- node{$\e_1$}(0011);
  \draw [bdedge] (0011)-- node{$\e_2$} (00011);

  \draw [bdedge] (101)-- node{$\e_0$} (0101);
  \draw [bdedge] (0101)-- node{$\e_1$} (00101);

  \draw [bdedge] (1001)-- node{$\e_0$} (01001);

  \draw [rdedge] (11)-- node{$\e_1$} (101);
  \draw [rdedge] (011)-- node{$\e_2$} (0101);
  \draw [rdedge] (0011)-- node{$\e_3$} (00101);

  \draw [rdedge] (101)-- node{$\e_2$} (1001);
  \draw [rdedge] (0101)-- node{$\e_3$} (01001);

  \draw [rdedge] (1001)-- node{$\e_3$} (10001);

\tikzset{every node/.style={font={\scriptsize,\bfseries}}}
\node at (0) [left] {(0)};
\node at (1) [left] {(1)};
\node at (11) [left,outer ysep=5pt] {(11)};
\tikzset{every node/.style={font=\scriptsize}}
\node at (01) [below] {(01)};
\node at (001) [below] {(001)};
\node at (0001) [below] {(0001)};
\node at (00001) [below] {(00001)};
  
\node at (011) [below,outer ysep=5pt] {(011)};
\node at (101) [above,outer ysep=5pt] {(101)};

\node at (0011) [below,outer ysep=7pt] {(0011)};
\node at (0101) [below,outer ysep=7pt] {(0101)};
\node at (1001) [above,outer ysep=7pt] {(1001)};

\node at (10001) [right,outer ysep=7pt] {(10001)};
\node at (01001) [right,outer ysep=7pt] {(01001)};
\node at (00101) [right,outer ysep=7pt] {(00101)};
\node at (00011) [right,outer ysep=7pt] {(00011)};

\node at (2.5,-9.5) {$\vdots$};

\end{tikzpicture}
\caption{The skeleton of $\Lambda_S$ (only the rank 0, 1 and 2 components are shown), where black edges have degree 1, blue edges have degree $(1,0)$ and red edges have degree $(0,1)$. Dashed lines are to indicate the level of each vertex only; they are not part of the graph. The label on each vertex is its characteristic function written as a finite binary sequence. The label of a root is in bold.}\label{F:Lambda_S}

\end{figure}
\begin{defn}\label{D:root}
  We say that a vertex $v\in\Lambda_S^0$ is a root if $v\Lambda_S=\{v\}$. 
  For each $\lambda\in\Lambda_S$ we say that the root of $\lambda$ is $\root{\lambda}:=[\card{\lambda}-1]$.
\end{defn}

\begin{rems}
  \begin{enumerate}
  \item Note that $v$ is a root if and only if $v=[n]$ for some
    $n \in \NN_{-1}$, i.e., $v$ is an object in $\ssc$ or
    $v=\emptyset$.
   \item For every $v \in \Lambda_S^0$ we have
    $\rank{v} \le \lev{v}+1$, with equality if and only if $v$ is a
    root.

  \item A root $u$ is an initial object in $\Lambda_k$ for some $k$, i.e., given any other vertex $v \in \Lambda_k^0$, we have that $\card{v\Lambda_Su}=1$.
  \item For each $k \in \NN_0$, $[k-1]$ is the unique root in the $k$-graph $\Lambda_k$, and it is the root of every $\lambda \in \Lambda_k$.
    \item For each $\lambda \in \Lambda_S$, $\root{\lambda}$ is the unique root of the same rank as $\lambda$.
\end{enumerate}
\end{rems}

The graph $\Lambda_S$ provides a further example of an SCS.

\begin{example}
  Define an SCS $(X_k,\delta_i)$ by $X_k:=\bigcup_{-1\le m \le k}\ell^{-1}(m)$ (for $k \in \NN_{-1}$) and $\delta_i:X_{k-1} \into X_k, u \mapsto \e_i(u)$ for $i,k \in \NN_0$ and $i\le k $. We then have innovation sets $D_k=\ell^{-1}(k)$ for $k\in \NN_{-1}$. This SCS is saturated.
\end{example}

  We will find it useful in the sequel to identify a vertex $v \in \Lambda_S^0$, i.e. a finite subset of $\NN_0$, with its characteristic function $\chi=\chi_v : \NN_0 \into \{0,1\}$. When $0 \le \lev{v}=n<\infty$, we view $\chi$ as a finite binary sequence $(\chi_0,\chi_1,\ldots,\chi_n)$ with $\chi_n=1$, and we write $(0)$ for the sequence of 0s (corresponding to $\emptyset$). So for example, we identify $\{0,2,4\}$ with $(10101)$.  Thus we also view $\Lambda_S^0$ as the set of finite binary sequences ending in 1 together with $(0)$, which we will think of as labels of subspaces and vectors in the sequel.

Thus the level and rank from Definition \ref{D:rank-lev} become
\begin{align*}
  \lev{\chi} &= \begin{cases} \max \{i: \chi_i =1 \} & \text{if } \chi \neq (0),\\
    -1 & \text{otherwise,}
  \end{cases}\\
  \rank{\chi} &= \card{\{i : \chi_i=1\}},
\end{align*}
respectively, and $\rank{\chi} -1 \le \lev{\chi}$ with equality if and only if $\chi$ is a root.

Moreover, from Definition \ref{D:root}, the root of $\chi$, denoted $\underline{\chi}$, is the root with the same rank: $\rank{\underline{\chi}} = \rank{\chi}$.  Explicitly this is the label which starts with $\rank{\chi}$ entries equal to $1$ on the left and all other entries are $0$.

For any $\chi$ with $\lev{\chi}=n$ and $\lev{\underline{\chi}} = m \;(=\rank{\chi} -1)$ there is a unique morphism $f: [m] \rightarrow [n]$ so that $\chi_i=0$ if and only if $i$ is not in the image of $f$. 

The following observation will be used in Section 8.

\begin{lemma}\label{lem:canon}
Given $u,v \in \Lambda^0_S$ such that $u \le v$,
\begin{align*}
& u = \{m_1< \ldots <m_r\} \;\text{with label}\; \chi, \\
& v = \{n_1< \ldots <n_r\},
\end{align*}
then there is a unique morphism $f \in \Delta_S$ with $s(f)=[m_r] = [\lev{u}],\; r(f)=[n_r] = [\lev{v}]$ so that $f(u)=v$ and
\[
f=\e_{j_t} \ldots \e_{j_1} \;\text{and}\; \chi_{j_1}=1, (\e_{j_{s-1}}\ldots \e_{j_1}(\chi))_{j_s}=1 \; \text{for all }\; 2 \le s \le t.
\]
\end{lemma}

We can think of this $f$ as a canonical representative in $\Delta_S$ of the unique morphism in $\Lambda_S$ from $u$ to $v$.
Here the labels play a role as control bits controlling which $\e_i$ we are allowed to use in the product.
The lemma can be proved by verifying that the conditions on $f$ lead to a unique solution which is
\[
f(k) = 
\left\{
\begin{array}{lll}
k & \mbox{if }  k < m_1 \\
n_i + (k-m_i) & \mbox{if }  m_i \le k < m_{i+1}, \; i=1, \ldots, r-1, \\
n_r & \mbox{if }  k = m_r.
\end{array}
\right.
\]

\section{Labeled Subspaces}{\label{S:labeled_subspaces}}

In this section we investigate the role of the graph $\Lambda_S$, as introduced in the previous section, for general SCHs. Recall that the vertices are given by finite subsets of $\Nset_0$. In this section we mainly use the notation based on characteristic functions which we refer to as labels $\chi$. 
Let $(\CH_k, \delta_i)_{k,i}$ be any SCH.

\begin{definition}
	If $\underline{\chi}$ is the root with $\lev{\underline{\chi}} = k$ then we define
	\begin{align*}
	L_{\underline{\chi}} 
    &:= \{ x \in \CH_k: x \perp \delta_i (\CH_{k-1})\quad \text{for}\;\; i=0,\ldots,k \} \\
    & \;= \{ x \in \CH_k: x \perp \alpha_i (\CH_{k-1})\quad \text{for all}\;\; i \in \Nset_0  \},
      \end{align*}
      in particular $L_{\underline{\chi}} = \CD_k$ for $k=-1$ and $k=0$. 
	We also refer to any $0 \not= x \in L_{\underline{\chi}}$ as a root vector
	of level $\lev{\underline{\chi}} = k$ of the SCH.
	
If $\chi$ is any label with $\underline{\chi}$
the corresponding root, i.e. $\rank{\underline{\chi}} = \rank{\chi}$, 
then there is a unique morphism f in $\Delta_S$ with domain $[\lev{\underline{\chi}}]$ and so that 
$\chi$ is the label of $im(f)$ (see Remark 5.4(3)).
Then the functor provides an isometry $\delta_f: \CH_{\lev{\underline{\chi}}} \rightarrow \CH_{\lev{\chi}}$. We define
\[
L_\chi := \delta_f L_{\underline{\chi}} \quad (\subset \CH_{\lev{\chi}}).
\]
We refer to the $L_\chi$ as labeled subspaces.
We have the partial isometry $\alpha_\chi: \CH_\infty \rightarrow \CH_\infty$ with initial space 
$L_{\underline{\chi}}$ and final space $L_\chi$ provided by the restriction of $\delta_f$.
We can then also write
\[
L_\chi = \alpha_\chi L_{\underline{\chi}}.
\]
\end{definition}

We state some immediate observations. For the root $\underline{\chi}$ with $\lev{\underline{\chi}}=k$ we can also write
\[
L_{\underline{\chi}} = \{ x \in \CD_k: x \perp \alpha_i (\CH_{k-1})\quad \text{for all}\; i \in \Nset_0  \}
\] 
(because for $i > k-1$ we have $\alpha_i(\CH_{k-1}) = \CH_{k-1}$ and hence
$L_{\underline{\chi}} \subset D_k$) and also
\[
L_{\underline{\chi}} = \{ x \in \CD_k: x \perp \alpha_i (\CD_{k-1})\quad \text{for all}\; i \in \Nset_0  \}
\]
(because $\CH_{k-1} = \bigoplus^{k-1}_{\ell=-1} \CD_\ell$ and for $\ell \le k-2$ we have $\alpha_i (\CD_\ell) \subset \CH_{k-1} \perp D_k$). 
In particular the labeled subspaces $L_{\underline{\chi}}$
of different roots are orthogonal to each other. 
Note that $x \in \CD_k$ is a root vector of level $k$ of the SCH, i.e.\! is in $L_{\underline{\chi}}$ for the root $\underline{\chi}$ of level $k$, if and only if it is orthogonal to all other labeled subspaces $L_\chi$ with level $\lev{\chi}\le k$. Hence $\CH_k$ is the closed linear span of the labeled subspaces $L_\chi$ with level $\lev{\chi} \le k$ and $\CH_\infty$ is the closed linear span of all labeled subspaces. This can be checked by induction, starting with 
$\CD_{-1} = L_0$ and $\CD_0 = L_1$. Then $x \in L_{11}$ if $x \in \CD_1$ and $x$ is orthogonal to $L_{01} = \alpha_0 (L_1)$, etc.
The latter is an instance of the general (and useful) formula
\[
\alpha_i L_\chi = L_{\epsilon_i({\chi})},
\]
where $\epsilon_i$ is the operation on labels which inserts $0$ at position $i$, so
\[
\epsilon_i(\chi)_j =
\left\{
\begin{array}{ll}
\chi_j & \mbox{if } j<i \\
0 & \mbox{if } j=i \\
\chi_{j-1} & \mbox{if } j>i 
\end{array}
\right.
\]
\begin{definition}
If $\chi$ and $\chi'$ are labels with the same rank, $\rank{\chi}=\rank{\chi'}$, then we define the partial isometry $\alpha_{\chi,\chi'} := \alpha_{\chi'} \alpha^*_\chi$.
The initial space is $L_\chi$, the final space is $L_{\chi'}$.
\end{definition}

For the root $\underline{\chi}$ of $\chi$ we have $\alpha_{\underline{\chi},\chi} = \alpha_\chi$. These partial isometries will play an important role and we will make use of the fact that whenever we have two labels with the same rank then we have such a (uniquely defined) partial isometry between the labeled subspaces. 

We can think of the construction of labeled subspaces as a functor from the category $\Lambda_S$ to the category of Hilbert spaces with isometries. Let us write $\delta_{\chi,\chi'}$ instead of $\alpha_{\chi,\chi'}$ if we think of these maps as surjective isometries from $L_\chi$ to $L_{\chi'}$ instead of partial isometries on $\CH_\infty$. Then the functor is the following:
For the objects we do $\chi \mapsto L_\chi$ and the image of the unique morphism in $\Lambda_S$ from $\chi$ to $\chi'$ (if it exists, i.e., if $\chi \le \chi'$) is $\delta_{\chi,\chi'}$. It is this functorial property which makes this construction useful.
\\

The relative position of labeled subspaces gives insights into the structure of an SCH.
For all SCHs, by the cosimplicial identities $\delta_j \delta_i = \delta_i \delta_{j-1}$, for $0 \le i < j \le k+1$ we always have

\[
\delta_j \,(\delta_i \CH_{k-1}) \quad\big( = \delta_i \delta_{j-1} \CH_{k-1} \big)\quad \subset \delta_i \CH_k
\]  

In the following we will see that the relative position of the labeled subspaces depends on additional information about adjoints.

\begin{lemma}\label{adjoint}
For a braided SCH with Artin generators $(\sigma_i)_{i \ge 1}$ we have
\[
\alpha_i^* \sigma_{j+1} = \sigma_j \alpha_i^* \quad \text{for} \; 0 \le i < j.
\]
For a saturated SCH
and $\delta_i: \CH_{k-1} \rightarrow \CH_k$
(with $0 \le i \le k$) we have $\delta_i^* = \alpha_i^*|_{\CH_k}$. 
Hence if the SCH is braided and saturated we also have
$\delta_i^* \sigma_{j+1} = \sigma_j \delta_i^* \quad \text{for} \; 0 \le i < j \le k+1$ (on $\CH_k$).
\end{lemma} 

\begin{proof}
The $\sigma_i$ are unitary, hence $\sigma_i^{-1} = \sigma_i^*$. Further note that the braid relations imply that $\sigma_j \sigma_{j+1} \sigma_j^{-1} = \sigma_{j+1}^{-1} \sigma_j \sigma_{j+1}$ for all $i$. Hence
\begin{align*}
\sigma_{j+1}^* \alpha_i 
& = \sigma_{j+1}^{-1} \sigma_{i+1} \sigma_{i+2} \ldots = \sigma_{i+1} \ldots \sigma_{j-1} \sigma_{j+1}^{-1} \sigma_j \sigma_{j+1}	\sigma_{j+2} \ldots \\
& = \sigma_{i+1} \ldots \sigma_{j-1} \sigma_j \sigma_{j+1} \sigma_j^{-1} \sigma_{j+2} \ldots
= \alpha_i \sigma_j^{-1} = \alpha_i \sigma_j^*.
\end{align*}
Taking adjoints yields the first result.
If the SCH is saturated then Theorem \ref{T:saturation} 
implies that $\alpha_i \CH_{k-1}^\perp \subset \CH_k^\perp$ and hence $\alpha_i^* \CH_k \subset \CH_{k-1}$. This shows that the adjoint of $\delta_i = \alpha_i |_{\CH_{k-1}}$ is $\delta_i^* = \alpha_i^* |_{\CH_k}$. 	
\end{proof}

This lemma seems to be the only part of this paper where we seriously use the fact that we have a (unitary) representation of the (full) braid group $\Bset_\infty$, instead of the braid monoid $\Bset^+_\infty$ considered in \cite{EGK17}.

\begin{definition}
An SCH $(\CH_k, \delta_i)_{k,i}$	is called normal if
\[
\delta_{j-1} \delta_i^* = \delta_i^* \delta_j: \CH_k \rightarrow \CH_k
\quad \text{for all} \;\; 0 \le i < j \le k+1.
\]
\end{definition}


The following diagram shows on the left the cosimplicial identities valid for all SCHs and on the right the additional identities valid only for normal SCHs.

\[
\begin{tikzcd}
& \text{SCH} & & \text{normal SCH} &\\
& \CH_k \arrow[r, "\delta_j"] & \CH_{k+1}   
& \CH_k \arrow[r, "\delta_j"]
\arrow[dl, "\delta_i^*" near start]
& \CH_{k+1} \arrow[dl, "\delta_i^*" near start] \\
\CH_{k-1} \arrow[ur, "\delta_i"] 
\arrow[r, "\delta_{j-1}"]	
& \CH_k \arrow[ur, "\delta_i"] & 
\CH_{k-1} \arrow[r, "\delta_{j-1}"]	& \CH_k & \\
\end{tikzcd}
\]

\begin{theorem} \label{thm:label}
Let $(\CH_k, \delta_i)_{k,i}$ be an SCH. 
\\
The following assertions are equivalent:
\begin{itemize}
	\item[(a)] $(\CH_k, \delta_i)_{k,i}$ is symmetric and saturated \\
	(i.e., the canonical SCH of a unitary representation of $\mathbb{S}_\infty$).
	\item[(b)] $(\CH_k, \delta_i)_{k,i}$ is braided and saturated \\
	(i.e., the canonical SCH of a unitary representation of $\mathbb{B}_\infty$).
	\item[(c)] $(\CH_k, \delta_i)_{k,i}$ is normal.
	\item[(d)] 
	$\delta_j (\CH_k \ominus \delta_i(\CH_{k-1})) \subset \CH_{k+1} \ominus \delta_i(\CH_k)$ (for all $0 \le i <j \le k+1 $).
	\item[(e)] The labeled subspaces $L_\chi$ are orthogonal to each other (for all $\chi \in \Lambda^0_S$).
\end{itemize}	
\end{theorem}

\begin{proof}
	$(a) \Rightarrow (b)$ is trivial because a unitary representation of $\Sset_\infty$ yields a unitary representation of $\Bset_\infty$ which factors through it.  
To prove $(b) \Rightarrow (c)$ we think of 
$\delta_{j-1} \delta_i^*$ and $\delta_i^* \delta_j$ as maps from $\CH_k$ to $\CH_k$ (for $0 \le i <j \le k+1$). 
Assuming $(b)$ we can apply Lemma
\ref{adjoint} to obtain
\[
\delta_{j-1} \delta_i^* = \sigma_j \ldots \sigma_k \delta_i^* = \delta_i^* \sigma_{j+1} \ldots \sigma_{k+1} = \delta_i^* \delta_j 
\]
$(c) \Rightarrow (d)$. Note that $\CH_k \ominus \delta_i \CH_{k-1}$ resp. $\CH_{k+1} \ominus \delta_i \CH_k$ are the kernels of $\delta_i^*: \CH_k \rightarrow \CH_{k-1}$ resp.
$\delta_i^*: \CH_{k+1} \rightarrow \CH_k$. Hence (d) says exactly that $\delta_j$ maps the first kernel into the second and this indeed follows from (c). 

$(d) \Rightarrow (e)$.
We prove orthogonality of the labeled subspaces $L_\chi$ by induction with respect to $\lev{\chi}$. 
If $\lev{\chi}=-1$ or $\lev{\chi}=0$ we consider $L_0 = \CD_{-1}$ and $L_1 = \CD_0$ which are orthogonal by definition. Assume, as the induction hypothesis, that all $L_\chi$ with $\lev{\chi} \le k-1$ are orthogonal to each other, where $k \ge 1$. 
Consider $L_\chi$ and $L_{\chi'}$ with $\lev{\chi}, \lev{\chi'} \le k$ and $\chi \not= \chi'$. If $\chi$ or $\chi'$ is a root then $L_\chi$ and $L_{\chi'}$ are orthogonal by the definition of labeled subspaces. Hence, to complete the proof, we can assume that both $\{\chi_0, \ldots \chi_k\}$ and $\{\chi'_0, \ldots \chi'_k\}$ contain both $0$'s and $1$'s. 
If there exists $0 \le i \le k$ such that $\chi_i = 0 = \chi'_i$ then there exists labels $\chi^i \not= \chi'^i$ with $\lev{\chi^i}, \lev{\chi'^i} \le k-1$ such that
\[
L_\chi = \alpha_i L_{\chi^i}, \quad
L_{\chi'} = \alpha_i L_{\chi'^i}.
\]
$L_{\chi^i}$ and $L_{\chi'^i}$ are orthogonal to each other by assumption. Because $\alpha_i$ is an isometry it follows that $L_{\chi}$ and $L_{\chi'}$ are orthogonal to each other as well.

It remains to consider the case that there is no $0 \le i \le k$ so that both $\chi_i =0$ and $\chi'_i =0$. Choose $i$ minimal with the property that $\chi_i=0$ or $\chi'_i=0$. Say $\chi_i=0$. Then $\chi'_i \not=0$ but there exists $i < j (\le k)$ so that $\chi'_j =0$. Hence there exist labels $\chi^i$ and $\chi'^j$ with $\lev{\chi^i}, \lev{\chi'^j} \le k-1$ such that
\[
L_\chi = \alpha_i L_{\chi^i}, \quad
L_{\chi'} = \alpha_j L_{\chi'^j}.
\]
Because $\chi'_i \not= 0$ and $i<j$ we also have $\chi'^j_i \not= 0$. From the induction 
hypothesis we conclude that
\[
L_{\chi'^j} \subset \CH_{k-1} \ominus \delta_i \CH_{k-2}.
\] 
We can now use (d) to infer that 
\[
L_{\chi'} = \alpha_j L_{\chi'^j} \subset \CH_{k} \ominus \delta_i \CH_{k-1},
\]
which is indeed orthogonal to
$L_\chi = \alpha_i L_{\chi^i} \subset \delta_i \CH_{k-1}$. This concludes the proof by induction.

$(e) \Rightarrow (a)$.
From the orthogonality of the labeled subspaces we get (for all $k \ge -1$) that
\[
\CH_k = \bigoplus_{\lev{\chi}\le k} L_\chi, \quad
\CD_k = \bigoplus_{\lev{\chi}= k} L_\chi,
\]
where $\bigoplus$ stands for orthogonal direct sums.
Hence $\alpha_{k+1}(\CD_\ell) \subset \CD_{\ell+1}$ if $\ell \ge k+1$. 
So, by Theorem \ref{T:saturation},
the fixed point space $\CH^{\alpha_{k+1}}$ is equal to $\CH_k$ and we conclude that the SCH is saturated. 
We now want to construct a unitary representation of $\Sset_\infty$ on $\CH_\infty$ 
so that the given SCH is canonical for this representation. Note that the elements of $\Sset_\infty$ can be realized as finite permutations (of $\Nset_0$) and hence it acts on the set $\Lambda^0_S$ of labels by permutations. 
Such permutations preserve the rank $\rank{\cdot}$ and hence we can consider the corresponding partial isometries on $\CH_\infty$. Explicitly, if $\pi \in \Sset_\infty$ acts on $\chi \in \Lambda^0_S$
to produce $\pi.\chi \in \Lambda^0_S$ then, making use of the orthogonality of labeled subspaces, we can define 
\begin{align*}
u_\pi: \CH_\infty = \bigoplus_{\chi \in \Lambda^0_S} L_\chi & \rightarrow \CH_\infty = \bigoplus_{\chi \in \Lambda^0_S} L_\chi \\
x = \bigoplus_{\chi \in \Lambda^0_S} x_\chi & \mapsto \bigoplus_{\chi \in \Lambda^0_S} \alpha_{\chi,\pi.\chi}\; x_\chi
\end{align*}
and then $\pi \mapsto u_\pi$ is a unitary representation of $\Sset_\infty$ on $\CH_\infty$.
For $n \ge 1$ let $\sigma_n$ be the transposition of position $n-1$ and $n$ and let us write $u_n$ instead of $u_{\sigma_n}$. Then it is easily checked that the operation $\epsilon_n$, the insertion of $0$ at position $n$ of a label, can be written as
\[
\epsilon_n(\chi) = \lim_{m\to\infty} \sigma_{n+1} \sigma_{n+2} \ldots \sigma_{m}(\chi)
\]
and hence 
\[
\alpha_n = \text{sot}-\lim_{m\to\infty} u_{n+1} u_{n+2} \ldots u_{m}.
\]
It is now clear that $(u_m)_{m\ge 1}$ is a Hessenberg factorization for the partial shifts $(\alpha_n)_{n\in\Nset_0}$ and the given SCH is the canonical SCH for the unitary representation $\pi \mapsto u_\pi$ of $\Sset_\infty$. This proves that we have a symmetric SCH. 
\end{proof}

Let us put together some useful formulas about labeled subspaces for a normal SCH (some of them already used in the previous proof) in the following corollary.


\begin{corollary} \label{cor:label}
For a normal SCH we have
\begin{itemize}
\item[(1)]	
$\CH_k = \bigoplus_{\lev{\chi}\le k} L_\chi$
\item[(2)]
$\CD_k = \bigoplus_{\lev{\chi}= k} L_\chi$
\item[(3)]
$\CH_\infty = \bigoplus_{\chi \in \Lambda^0_S} L_\chi$
\item[(4)] 
$\alpha_i(\CH_\infty) \:\;= \bigoplus_{\chi:\; \chi_i=0} L_\chi = \{ x \in \CH_\infty: \alpha_i(x) = \alpha_{i+1}(x)\}$
\item[(5)] 
$\alpha_i(\CH_\infty)^\perp = \bigoplus_{\chi:\; \chi_i=1} L_\chi = \{ x \in \CH_\infty: \alpha_i(x) \perp \alpha_{i+1}(x)\}$
\item[(6)]
an operational characterization of labeled subspaces:
\begin{align*}
L_\chi  & = \bigcap_{i: \chi_i=0} \alpha_i(\CH_\infty) \cap 
\bigcap_{i: \chi_i=1} \alpha_i(\CH_\infty)^\perp \\
 & =  \left\{
 x \in \CH_\infty: \quad
\left.
\begin{array}{ll}
\alpha_i(x) = \alpha_{i+1}(x) & \mbox{for }  \chi_i =0 \\
\alpha_i(x) \perp \alpha_{i+1}(x) & \mbox{for }  \chi_i =1
\end{array}
\right.
\right\}
\end{align*}
\end{itemize}	
\end{corollary}

To obtain the second half of $(4),(5),(6)$
we can argue as follows. Let $x \in L_\chi$. If $\chi_i =0$ then $x = \alpha_i(y)$ for some $y \in \CH_\infty$, hence $\alpha_{i+1}(x) = \alpha_{i+1} \alpha_i(y) = \alpha_i \alpha_i(y) = \alpha_i(x)$.
(So this part even works without normality.) If $\chi_i =1$ then $\epsilon_i(\chi) \not= \epsilon_{i+1}(\chi)$. So $\alpha_i(x) \in L_{\epsilon_i(\chi)}$ and $\alpha_{i+1}(x) \in L_{\epsilon_{i+1}(\chi)}$ lie in different labeled subspaces and these are orthogonal, by normality, Theorem \ref{thm:label}(e).
\\
\begin{example}
As an easy and instructive example of an SCS which is saturated but not normal (hence not braided or symmetric) we can just remove the root element (the element that is not in the image of any $\delta_i$) from the rank $2$-graph $\Lambda_2$ (see Figure \ref{F:Lambda_S} in Section 5), to obtain an SCS 
whose initial part is sketched in the left part of Figure \ref{F:SCS_ex}.

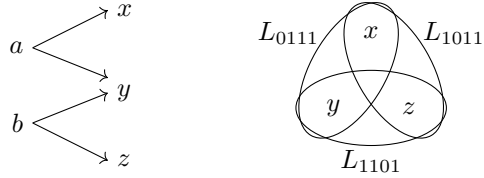
\begin{figure}[h]
  \begin{center}
    \begin{tikzpicture}
      \draw[->] (0,1.5) node[left] {$a$} -- (1,2) node[right] {$x$};
      \draw[->] (0,1.5) -- (1,1.1) ;
      \draw[->] (0,0.5) node[left] {$b$} -- (1,0.9) node[right] {$y$};
      \draw[->] (0,0.5) -- (1,0) node[right] {$z$};

      \draw[rotate around={60:(4.2,1.2)}] (4.2,1.2) ellipse (1cm and 0.5cm);
      \draw[rotate around={120:(4.8,1.2)}] (4.8,1.2) ellipse (1cm and 0.5cm);
      \draw[rotate around={0:(4.5,0.7)}] (4.5,0.7) ellipse (1cm and 0.5cm);
      \draw (4.5,1.7) node {$x$};
      \draw (4.0,0.7) node {$y$};
      \draw (5.0,0.7) node {$z$};
      \node at (3.4,1.7) {$L_{0111}$};
      \node at (5.6,1.7) {$L_{1011}$};
      \node at (4.5,0.0) {$L_{1101}$};
    \end{tikzpicture}
  \end{center}
  \caption{A non-normal, saturated SCS.}\label{F:SCS_ex}
\end{figure}
In this SCS now $a, b$ are root elements and the labels of all non-trivial labeled subspaces have rank $3$.  
We have $X_{-1}=X_0=X_1=\emptyset$, $X_2 = \{a,b\}$,
so $L_{111} = \spn\{a,b\}$.
We find further that $L_{0111} = \spn\{x,y\}, L_{1011} = \spn\{x,z\}, L_{1101} = \spn\{y,z\}$. The non-orthogonality of these spaces implies the non-normality of the SCH, by Theorem \ref{thm:label}(e). It is an instructive exercise to verify directly that the cohomology is still trivial, as it has to be by Corollary \ref{sat-c}.

We look at SCSs in more detail in Section 8. 
\end{example}

\section{Normal SCHs and tame representations of the infinite symmetric group $\mathbb{S}_\infty$}
Given a normal SCH we can consider its collection $\{L_{[n]}\}_{n \ge -1}$ of root spaces (of level $n$ and rank $n+1$). If for each $L_{[n]}$ we choose an ONB $\{e_{n,d}\}_{n \ge -1,d=1,\ldots,d(n)=dim L_{[n]}}$ then we obtain an ONB of $\CH_\infty$ by $\{\delta_f e_{n,d}\}_{f,n,d}$, where $f$ runs through all morphisms of $\Delta_S$ with domain $[n]$.
We call this an adapted ONB for the normal SCH. The following proposition is an immediate consequence of the existence of adapted ONBs. 
\begin{prop} \label{prop:complete}
Every normal SCH contains underlying SCSs given by adapted ONBs. The sequence $(d_n = dim L_{[n]})_{n \ge -1}$ is a complete invariant for normal SCHs with respect to unitary equivalence. In other words, every sequence $(d_n)_{n \ge -1}$ of cardinal numbers yields a normal SCH by choosing $dim L_{[n]} = d_n$ for all $n$, and normal SCHs are unitarily equivalent if and only if these sequences are the same. 
\end{prop}

The theory of normal SCHs is closely related to the theory of tame representations of the infinite symmetric group $\mathbb{S}_\infty$. Recall that a unitary representation $\pi: \mathbb{S}_\infty \rightarrow \BH$ is called tame if one (and hence both) of the following equivalent conditions are satisfied: 
\begin{itemize}
\item[(1)] 
$\pi: \mathbb{S}_\infty \rightarrow \BH$ is continuous, 
\\
where $\mathbb{S}_\infty$ is considered as a topological group with a neighborhood basis of the identity given by the subgroups $\langle \sigma_k, k \ge n \rangle$ for $n \in \Nset$ (which is the topology of pointwise convergence for the standard action on $\Nset$) and $\BH$ is equipped with the strong operator topology sot (or the weak operator topology wot which restricted to unitaries is the same).
\item[(2)]
The union of the increasing tower of fixed point spaces $\CH^{\pi(\sigma_k), k \ge n}$ over $n \in \Nset$ is dense in $\CH$.
\end{itemize}

See \cite{Ol85,Ok99}. The terminology `tame' for these representations is taken from \cite{Ok99}, it is not yet used in \cite{Ol85}. Unitary tame representations of $\mathbb{S}_\infty$ are classified by Lieberman's theorem \cite{Li72}: they can be decomposed into sums of irreducibles and each irreducible can be constructed as an induced representation from an irreducible representation of the finite symmetric subgroup $S_n$, for some $n \in \Nset$. In \cite{Ol85} this is derived by a semigroup approach. In the following we show that an alternative 
proof of Lieberman's theorem can be obtained based on the observation that the underlying Hilbert space carries the structure of a normal SCH, giving another reason for their study.  

We have seen in Theorem \ref{thm:label} that a normal SCH is symmetric, so its tower $(\CH_n)$ can be interpreted as the tower $(\CH_n = \CH^{\pi(\sigma_k), k \ge n+2})$ for a unitary tame representation $\pi$ of $\mathbb{S}_\infty$. 
Conversely, given a unitary tame representation $\pi: \mathbb{S}_\infty \rightarrow \BH$, we obtain the canonical SCH by putting $\CH_\infty := \CH, \;\; \CH_n = \CH^{\pi(\sigma_k), k \ge n+2}, \; \alpha_i = \text{sot}-\lim \pi(\sigma_{i+1} \sigma_{i+2} \ldots)$. This is normal by Theorem \ref{thm:label}. It is an elementary and useful observation that, for all $n$, (the fixed point space) $\CH_n$ carries a representation of the finite group $S_{n+1}$, obtained by restriction. Let us further compute explicitly how the generators $\sigma_j$ interact with the partial shifts $\alpha_i$.

\begin{lemma}\label{lem:sigma}
Using the notation above we have
\[
\pi(\sigma_j) \alpha_i = 
\left\{
\begin{array}{llll}
\alpha_i \pi(\sigma_{j-1}) & \mbox{if }  i < j-1, \\
\alpha_j & \mbox{if }  i = j-1, \\
\alpha_{j-1} & \mbox{if }  i = j, \\
\alpha_i \pi(\sigma_j) & \mbox{if }  i > j. \\
\end{array}
\right. 
\]
\end{lemma}
\begin{proof}
The case $i<j-1$ is a restatement of Lemma \ref{lem:braided}(1). In the case $i=j-1$ we get $\pi(\sigma_j) \alpha_i =
\pi(\sigma^2_j) \alpha_{i+1}$ from the definition of $\alpha_i$ and the result follows because $\sigma^2_j$ is the identity in $\mathbb{S}_\infty$. The case $i=j$ is immediate from the definition of $\alpha_i$. Finally the case $i>j$ follows from the definition of $\alpha_i$ together with the fact that $\sigma_k \sigma_j = \sigma_j \sigma_k$ if $k>i$.
\end{proof}

From this we can determine how the representation interacts with the labeled subspaces. 

\begin{proposition} \label{prop:sigma}
Given a unitary tame representation $\pi$ of $\mathbb{S}_\infty$.
If $x \in L_\chi$ then $\pi(\sigma_j)x \in L_{\chi'}$,
where the label $\chi'$ is obtained from the label $\chi$ by the transposition $(j-1,j)$. In particular $\pi$ preserves the rank of a label.   
\end{proposition}

\begin{proof}
Recall from Corollary \ref{cor:label} that 
$x \in L_\chi$ 
if and only if $x$ is in the range of the partial shift $\alpha_i$ whenever $\chi_i=0$ and orthogonal to the range of $\alpha_i$ whenever $\chi_i=1$. Now an inspection of Lemma \ref{lem:sigma} shows that 
$\pi(\sigma_j)x \in L_{\chi'}$ (take into account that the $\pi(\sigma_j)$ are unitary). Hence $\pi(\sigma_j)$
preserves the rank of a label and because the $\sigma_j$
generate $\mathbb{S}_\infty$ this is true for all the unitaries in the representation $\pi$.
\end{proof}

If $\CK$ is any $\mathbb{S}_\infty$-invariant subspace for a tame unitary representation on $\CH$ then the subrepresentation is tame again (this is obvious from the characterization by continuity) and thus this is again a normal SCH. We see from Corollary \ref{cor:label}(6) that labels do not change if we go to a subrepresentation. Now we see from Proposition \ref{prop:sigma} that, if necessary, we can decompose the representation further and assume from now on that only vectors of a fixed rank $m+1$ appear in $\CK$. Then there is a uniquely determined root space $\CK_m = \CH_m \cap \CK$, cyclic for the representation. Note that $\CK_m$ is an $S_{m+1}$-invariant subspace of $\CH_m$. 
\\
Conversely, any 
$S_{m+1}$-invariant subspace of $\CH_m$ arises in this way. Indeed, starting from such a subspace of root vectors of rank $m+1$ we can form the corresponding normal SCH and, as can be checked by Lemma \ref{lem:sigma}, this produces an $\mathbb{S}_\infty$-invariant subspace $\CK$ of $\CH$. It is characterized among other $\mathbb{S}_\infty$-invariant subspaces, as a normal SCH, by its root vectors. 
The conclusion is that there is a bijective correspondence between $S_{m+1}$-invariant subspaces of $\CH_m$ and $\mathbb{S}_\infty$-invariant subspaces of $\CH$, by extension respectively by restriction. But
$S_{m+1}$ is a finite group and any representation can be decomposed into irreducible representations. Hence the same is true for tame unitary representations of $\mathbb{S}_\infty$.

All that remains to check for a full proof of Lieberman's theorem is that a unitary tame representation $\pi$ of $\mathbb{S}_\infty$ is an induced representation. We can again assume that only vectors of a fixed rank $m+1$ appear, so we have a uniquely determined root space $\CH_m$. Because $\CH_m = \CH^{\pi(\sigma_k), k \ge m+2}$ this space carries,
by restriction of $\pi$, a representation of $\langle \sigma_k: k \in \Nset \setminus \{m+1\} \rangle$ where
$\langle \sigma_k: k \ge m+2 \rangle$ acts identically. We claim that $\pi$ is induced by that. By 1.F.1 in \cite{BdlH20} it is enough to check that for a left transversal $T$ of (the open subgroup) $\langle \sigma_k: k \in \Nset \setminus \{m+1\} \rangle$ in $\mathbb{S}_\infty$ the spaces $\pi(t) \CH_m$ for $t \in T$ are all orthogonal to each other. For $m < n$ we can consider 
$S_{n+1} = \langle \sigma_k: 1 \le k \le n \rangle$ as permuting $[n]$, so we can identify a left coset of
$\langle \sigma_k: k \in [n] \setminus \{m+1\} \rangle$ with a subset of size $m+1$ in $[n]$, namely the image of
$[m]$ under any permutation in the coset. Letting $n$ vary we can identify any $t \in T$ with a subset $A_\chi$ of size $m+1$ of $\Nset_0$ (a label $\chi$ of rank $m+1$). Explicitly, for $A_\chi \subset [n]$ (that is, labels $\chi$ with level not exceeding $n$), we can choose 
\[
t=\sigma_\chi :=
\sigma_{i_k}^{[n]}\sigma_{i_{k-1}}^{[n-1]}\cdots \sigma_{i_1}^{[m+1]}, \text{ with } 0 \le i_1< i_2 <\cdots < i_k\le n,
\]
where $A_\chi = [n]\setminus \{i_1,i_2,\ldots,i_k\}$, $k:=n-m$, and $\sigma_p^{[l]} := \sigma_{p+1} \ldots \sigma_l \text{ for all } 0\le p \le l $ (with $\sigma_l^{[l]}$ to be interpreted as identity). 
We can now recover the labeled subspaces by $L_{\underline{\chi}} = \CH_m$ and
\[
L_\chi = \alpha_\chi L_{\underline{\chi}} = \pi(\sigma_\chi) \CH_m.
\]
But the orthogonality of labeled subspaces follows from Theorem \ref{thm:label}. This concludes our proof of Lieberman's theorem. We see that the labeled subspaces form a kind of system of imprimitivity associated to an induced representation of the (not locally compact) group $\mathbb{S}_\infty$ (see for example \cite{Ma78} for the classical version of systems of imprimitivity for locally compact groups). 

Reflecting on these arguments it is possible to come to the conclusion that the theory of normal SCHs is just another way of looking at the theory of tame representations of $\mathbb{S}_\infty$. And in fact it is a relatively elementary way of approaching these. But it is important to remind us that normal SCHs are also produced as canonical SCHs by representations of the infinite braid group $\mathbb{B}_\infty$ which have an analogous continuity property. A lot less is known about these representations and the situation is in fact much more complicated (compare Lemma \ref{lem:sigma} from which only parts are still valid for $\mathbb{B}_\infty$). Still we may hope that the toolkit of SCHs can lead to some progress here as well. We don't go further into it in this paper, these questions are currently under investigation. 

\section{Classification of SCSs and Normal Extensions}\label{S:SCS_ext}
Let $(Y_k,\delta_i)_{k,i}$ be an SCS. Recall that by thinking of each $Y_k$ as an orthonormal basis of $\CH_k$ of an SCH we can think of an SCS as a special kind of SCH where the $\delta_i$ map basis vectors (injectively) to basis vectors and terminology and results about SCHs apply to SCSs. The partial shifts $\alpha_i$ provide injective maps on $Y_\infty = \bigcup^\infty_{k=-1} Y_k$ which we still call $\alpha_i$.
By elements and subsets of the SCS we mean elements and subsets of $Y_\infty$.

In particular for an SCS we have (disjoint) innovation subsets and labeled subsets. `Orthogonal' for subspaces becomes `disjoint' if we talk about subsets. We notice from the definitions that every $y \in Y_\infty$ belongs to at least one labeled subset and in fact, by Theorem \ref{thm:label}(e), we have a partition of $Y_\infty$ into disjoint labeled subsets (refining the partition into innovation subsets) if and only if the SCS is normal. Hence every $y \in Y_\infty$ in a normal SCS has a unique label. We can use Corollary \ref{cor:label}(6) to motivate the following definition and use this as a clue how to construct normal extensions of an SCS. This also provides a way to classify SCSs in general. 

\begin{definition}
The normal label $\hat{\chi}(y) \in \Lambda_S^0$ for any $y \in Y_\infty$ (in any SCS $(Y_k,\delta_i)_{k,i}$) is defined for $n \in \Nset_0$ by
\[
\hat{\chi}(y)_n := 
\left\{
\begin{array}{ll}
0 & \mbox{if } \alpha_n(y) = \alpha_{n+1}(y) \\
1 & \mbox{if } 
\alpha_n(y) \not= \alpha_{n+1}(y)
\end{array}
\right.
\]
\end{definition}
We see from Corollary \ref{cor:label}(6) that if the SCS $(Y_k,\delta_i)_{k,i}$ is normal then the normal label $\hat{\chi}(y)$ is exactly the label of the uniquely labeled subset to which $y$ belongs. On the other hand, because the actions of the partial shifts on $y$ do not change if we go to an extension, it follows that even in a non-normal SCS $\hat{\chi}(y)$ is the correct label for $y$ and its labeled subset in any normal extension (if it exists). This motivates the terminology `normal label'.
Note that $\hat{\chi}(y)_n=1$
only for finitely many $n$ because $\alpha_m(y)=y$  for all $m$ large enough and hence indeed $\hat{\chi}(y) \in \Lambda_S^0$. We can also see from the normal labels if an SCS is saturated: this is the case if and only if for all $y \in Y_\infty$ we have $y \in Y_{\lev{\hat{\chi}(y)}}$. In fact, if we produce the saturation of an SCS then the set $Y_\infty$ does not change, we only move each $y \in Y_\infty$ into $Y_{\lev{\hat{\chi}(y)}}$ (if not already there) and hence achieve
$Y_{\lev{\hat{\chi}(y)}} = Y^{\alpha_{\lev{\hat{\chi}(y)}+1}}_\infty$.

\begin{lemma} \label{lem:normlab}
$\hat{\chi}\big(\alpha_j(y)\big) = \epsilon_j(\hat{\chi}(y))$ (for all $y \in Y_\infty$ and $j \ge 0$).
\end{lemma}

\begin{proof}
Let $i,j \ge 0$. 

1) $\quad  \alpha_j\big(\alpha_j(y)\big) = \alpha_{j+1}\big(\alpha_j(y)\big)$,
hence $\hat{\chi}\big(\alpha_j(y)\big)_j = 0 = \epsilon_j(\hat{\chi}(y))_j$. 

2) If $i<j$ then
\begin{align*}
\alpha_i\big(\alpha_j(y)\big) &= \alpha_{j+1}\big(\alpha_i(y)\big) \\
\alpha_{i+1}\big(\alpha_j(y)\big) &= \alpha_{j+1}\big(\alpha_{i+1}(y)\big),
\end{align*}
hence (by injectivity of $\alpha_{j+1}$)
\[
\hat{\chi}\big(\alpha_j(y)\big)_i = \hat{\chi}(y)_i = \epsilon_j(\hat{\chi}(y))_i.
\]

3) If $i>j$ then
\begin{align*}
\alpha_i\big(\alpha_j(y)\big) &= \alpha_j\big(\alpha_{i-1}(y)\big) \\
\alpha_{i+1}\big(\alpha_j(y)\big) &= \alpha_j\big(\alpha_i(y)\big),
\end{align*}
hence (by injectivity of $\alpha_j$)
\[
\hat{\chi}\big(\alpha_j(y)\big)_i = \hat{\chi}(y)_{i-1} = \epsilon_j(\hat{\chi}(y))_i.
\]
\end{proof}

Combining the previous discussion with this lemma it is now possible to establish that each SCS has a normal extension and in fact to give a structure theorem for SCSs which classifies them completely, up to isomorphism. Clearly it is enough to consider only saturated SCSs (the general case obtained by `moving into wrong boxes', see Example \ref{E:basic_examples}) which allows us to focus on the elements and the way the partial shifts act on them (the tower determined by this as fixed point sets). A root element is an element not in the image of any coface map $\delta_i$. Root elements exist in every nonempty SCS, for example the lowest index non-empty innovation set consists of root elements. Every element can be written as $\delta_g(\underline{y})$ for a root element $\underline{y}$ and a morphism $g \in \Delta_S$.
After these elementary preparations we now put the results together in the following theorem.

\begin{theorem}
In the following all SCSs are considered to be saturated. 
\begin{itemize}
\item[(a)]
Let $\Sset$ be an SCS with a single root element $\underline{y}$ of normal rank $\rank{\hat{\chi}(\underline{y})}=k$. Then $\Sset$ is isomorphic (as an SCS) to the sub-SCS $\hat{\chi}(\underline{y})^{\le} \subset \Lambda_k$ (i.e., all labels in $\Lambda_k$ which dominate $\hat{\chi}(\underline{y})$, see Section 5). Two of them, $\Sset$ resp. $\Sset'$ with single root elements $\underline{y}$ resp. $\underline{y}'$ of normal rank $k$, are isomorphic as SCSs if and only if these root elements have the same normal label: $\hat{\chi}(\underline{y}) = \hat{\chi}(\underline{y}')$.
$\Sset$ is normal if and only if $\hat{\chi}(\underline{y}) = (1, \ldots, 1, 0, 0, \ldots)$, i.e., if the normal label of the root element is the root of $\Lambda_k$. In other words, $\Sset$ is normal if and only if it is isomorphic to $\Lambda_k$ (as an SCS). 
As a poset all these SCSs are isomorphic to $\Lambda_k$. 

\item[(b)]
A normal SCS $\Sset$ is a disjoint union of normal SCSs with single root elements
(as classified in (a), we call each of them a layer of $\Sset$).
A complete invariant for normal SCSs, up to isomorphism, is given by the multiplicities 
for layers of the form $\Lambda_k$
within $\Sset$.
\item[(c)]
  Any SCS has an extension to a normal SCS (in other words: is a sub-SCS of a normal SCS).  There is a minimal normal extension which is unique up to isomorphism.

\item[(d)]
A complete invariant for SCSs, up to isomorphism, is given by the invariant described in (b) applied to the minimal normal extension together with, for each layer in the minimal normal extension, a nonempty finite set of labels with none of them dominating any of the others. These labels are the normal labels of the root elements of the SCS. 
\end{itemize}
\end{theorem}

\begin{proof}
(a)
We want to show that we can identify an element $y$ (of the SCS $\Sset$ with a single root element $\underline{y}$) with its normal label, that is, any two elements with the same normal label are the same. We write $y=\delta_g(\underline{y})$ for a morphism $g \in \Delta_S$. By Lemma \ref{lem:normlab}
we see that $g = \e_{j^\prime_t} \ldots \e_{j^\prime_1}$ maps the normal label of $\underline{y}$ to the normal label of $y$. We now successively replace (if necessary) the $j^\prime_s$ by the smallest $j_s \ge j^\prime_s$ so that $f = \e_{j_t} \ldots \e_{j_1}$ satisfies the conditions of Lemma \ref{lem:canon}.
Hence in the end we have replaced $g$ by the unique canonical morphism $f$ of Lemma \ref{lem:canon}.
This doesn't change the action on the labels. 
In fact we see from the definition of the normal labels that also $\alpha_{j^\prime_s}$ acts in the same way as $\alpha_{j_s}$ on the element obtained at this point. 
(Similarly, note that the replacement of $j^\prime_s$ by $j_s$ above is not possible only in the case where both
$\e_{j^\prime_s}$ and $\alpha_{j^\prime_s}$ act identically on the label respectively the element obtained at this point, so these can be safely removed.)
Hence $y=\delta_g(\underline{y}) = \delta_f(\underline{y})$, so $y$ is uniquely determined by its normal label. 

Given $\hat{\chi}(\underline{y})$ as the normal label of a single root element $\underline{y}$, 
Lemma \ref{lem:normlab} shows that the partial shifts of $\Sset$ act on the normal labels in exactly the same way as the corresponding morphisms do in $\Lambda_k$, so we can identify $\Sset$ with 
$\hat{\chi}(\underline{y})^{\le}$ (as an SCS) by identifying elements of $\Sset$ with their normal labels. 
If $\hat{\chi}(\underline{y}')$ is different it can be seen from the definition of normal labels that now some of the actions of partial shifts are different and what we obtain is not isomorphic (as an SCS). The actions in the normal case are only obtained if $\hat{\chi}(\underline{y})$ is also the root in $\Lambda_k$. As a poset all these are isomorphic because here, without changing the poset, we can remove all $0's$ to the left of $1's$ in $\hat{\chi}(\underline{y})$, arriving at the root of $\Lambda_k$.
\\
(b)
If two different elements $y$ and $y'$ of a normal SCS $\Sset$ have the same label then, by injectivity of the partial shifts, the corresponding roots $\underline{y}$ and $\underline{y}'$ must also be different. So we obtain different layers. In fact we are just repeating here the argument in the beginning of Section 7 when we introduced adapted bases and the multiplicities here are the dimensions of the root spaces in Section 7. Note that for a normal SCH the adapted bases are underlying normal SCSs and any normal SCS can be thought of in that way.
\\
(c)
If $x$ is an element of an SCS then we denote by $x^{\le}$ the set of all $\delta_f(x)$, it is in one-to-one correspondence with the labels dominating $\hat{\chi}(x)$, using the argument shown for (a). Note that two labels of the same rank always have a common upper bound, hence by defining, for $x,y$ in the SCS, that
$x \sim y$ if and only if $x^{\le} \cap y^{\le} \not= \emptyset$ we obtain an equivalence relation. Suppose $x,y$ are in the same equivalence class and $y$ has the same normal label as an element
$x^\prime \in x^{\le}$. Then for some morphism $f$ also $\delta_f(y) \in x^{\le}$ and the latter element has the same normal label as $\delta_f(x^\prime)$. Therefore $\delta_f(y) = \delta_f(x^\prime)$ and, by injectivity, $y=x^\prime$. We conclude that if $x,y$ are in the same equivalence class and $x \not= y$
then also $\hat{\chi}(x) \not= \hat{\chi}(y)$. This allows us to embed the whole equivalence class (in a faithful way) into a normal SCS with a single root element (of the corresponding rank). For different equivalence classes we can use different normal layers as described above. It is clear from this construction that any normal extension must contain what we have obtained from this procedure. So this gives us a unique minimal normal extension. 
\\
(d)
It remains to show: In any layer (normal SCS with a single root element) the number of root elements of a sub-SCS is always finite. We can think of the layer as $\Lambda_k$ and make use of its representation as $\Nset^k_0$, as explained in Section 5.
We show that there is no infinite sequence $\big(m^{(j)} = (m^{(j)}_1,\ldots,m^{(j)}_k)\big)^\infty_{j=1} \subset \Nset^k_0$ with the property that whenever $i<j$ then there exists $p_{ij} \in \{1,\ldots,k\}$
so that $m^{(i)}_{p_{ij}} > m^{(j)}_{p_{ij}}$.
Proof by induction in $k$. The case $k=1$ is obvious. Suppose the claim is true for $k-1$ but, to get a contradiction, that there is such an infinite sequence for $k$. Then there is an infinite subsequence
indexed by $(j_\ell)$ so that $p_{1j_\ell}=:p$ is constant and also $m^{(j_\ell)}_p =:m$ is constant. If we remove the $p$-th entry $m$ from this subsequence we obtain an infinite sequence of $(k-1)$-tuples satisfying the property. This is a contradiction. 
\end{proof}

The corresponding extension problem for SCHs (instead of SCSs) is related to the braidability problem mentioned in the introduction (see \cite{GK09} for the original formulation) because normal SCHs carry braid group representations, by Theorem \ref{thm:label}(b). But the extension problem for SCHs presents additional difficulties because in an SCH not all vectors have labels, they are just linear combinations of vectors in labeled subspaces. So far we have not been able to fully solve the problem: Does any SCH always have a normal extension?

\section{Spreadability with respect to an inner product}

The study of SCHs is motivated by the study of spreadable sequences in (non-commutative) probability theory, see the introduction and \cite{EGK17}. We can make this more explicit by introducing the corresponding Hilbert space version of spreadability. As we will see this is also a very concrete way of thinking about (parts of) SCHs arising from roots of level $0$.
We say that a sequence $(x_n)_{n \ge 0}$ of unit vectors in a Hilbert space is spreadable with respect to the inner product if whenever $i < j$ then $\langle x_j, x_i \rangle$ does not depend on $i$ and $j$. Intuitively this provides a geometric figure with a certain symmetry based on some fixed angle. It is useful to generalize this idea to a higher dimensional version.

\begin{definition} 
Let $\CK, \CH$ be Hilbert spaces. A sequence $(\iota_n)_{n \ge 0}$ of isometries from $\CK$ into $\CH$ is called spreadable (with respect to the inner product) if $\iota_j^* \iota_i = C$ if $i < j$
(independent of $i,j$). Here $C: \CK \rightarrow \CK$ is a contraction which we call the operator angle of the spreadable sequence $(\iota_n)_{n \ge 0}$.
\end{definition}

Note that if we put $\CK := \Cset$ and $x_n := \iota_n(1)$ for all $n$ then we recover the notion of a spreadable sequence of (unit) vectors. 

It is immediate how this is related to SCHs.
If $(\iota_n)_{n \ge 0}$ is a spreadable sequence of isometries from $\CK$ into $\CH$
then we can define $\CH_k$ to be the closed linear span of the $\iota_\ell(\CK)$ with $\ell \le k$ and 
\[
\alpha_n(x) := 
\left\{
\begin{array}{ll}
x & \mbox{if } x \in \iota_\ell(\CK),\; \ell < n \\
\iota_{\ell+1} \,\iota_\ell^*(x) & \mbox{if } 
x \in \iota_\ell(\CK), \;\ell \ge n
\end{array}
\right.
\]
and extend linearly. In fact, if $x = \sum \iota_\ell(x_\ell)$ with $x_\ell \in \CK$ then $\alpha_n(x) = \sum_{\ell < n} \iota_\ell(x_\ell) + \sum_{\ell \ge n} \iota_{\ell+1}(x_\ell)$ and from spreadability we obtain $\| \alpha_n(x)\| = \|x\|$. This shows that $\alpha_n$ is a well defined isometry on $\CH_\infty$ (for all $n \in \Nset_0$). Similarly it can be checked that the $\alpha_n$ are partial shifts. We obtain a corresponding SCH with $\CH_0 = \iota_0(\CK)$ as roots of level $0$ and no other roots of higher level. We call this the minimal SCH for the spreadable sequence. 

Conversely, if we have any SCH then we can define $\CK := \CH_0$ and $\iota_n := \alpha_0^n|_{\CH_0}$ and obtain a spreadable sequence. In fact, if $i<j$ and $x,y \in \CK$ then we get (using $\alpha_1 \alpha_0 = \alpha_0 \alpha_0$ and $\alpha_1(y)=y$)
\begin{align*}
& \langle x, \iota_j^* \iota_i y \rangle  = \langle \iota_j x, \iota_i y \rangle = \langle \alpha_0^j(x), \alpha_0^i(y) \rangle = \langle \alpha_0^{j-i-1} \alpha_0(x), y \rangle \\
& = \langle \alpha_1^{j-i-1} \alpha_0(x), \alpha_1^{j-i-1}(y) \rangle = \langle \alpha_0(x), y \rangle = \langle \iota_1(x), \iota_0(y) \rangle = \langle x, \iota_1^* \iota_0 y \rangle.
\end{align*}
which proves spreadability. This is analogous to the 'constructive procedure' to produce spreadable random variables discussed in \cite{GK12}.

Together this shows that the theory of spreadability with respect to an inner product is just a concrete version of the theory of SCHs where all roots have level $0$ (other parts of an SCH are ignored by this approach). 
\\

We give an example where everything can be easily computed explicitly.
\\
Consider on the Hilbert space 
$\CH := \ell^2 (n \in \Zset: -1 \le n < \infty)$
the partial shifts $(\alpha_n)_{n\in \Nset_0}$
given by
\[
\alpha_n (\beta_{-1}, \beta_0, \ldots, \beta_{n-1}, \beta_n, \ldots) = (\beta_{-1}, \beta_0, \ldots, \beta_{n-1}, 0, \beta_n, \ldots).
\]
We obtain an SCH (with all roots in level $0$)
by choosing a subspace $\CH_0 \subset \CH^{\alpha_1} = \{ (\beta_{-1}, \beta_0, 0, 0, \ldots) \}$ and then $\CH_k$ as the closed linear span of the subspaces $\alpha_0^\ell \CH_0$ with $\ell \le k$ (which is contained in $\CH^{\alpha_{k+1}} = \{ (\beta_{-1}, \beta_0, \ldots, \beta_k, 0, 0, \ldots) \}$).
Let us choose $\CH_0 := \Cset x_0$ with
$x_0 := \frac{1}{\sqrt{2}} (1,1,0,0, \ldots)$.
Then with $\CK := \CH_0$ we obtain a spreadable sequence $(x_n)_{n \in \Nset_0}$
where $x_n := \alpha_0^n(x_0)$. We have 
$x_1 = \frac{1}{\sqrt{2}} (1,0,1,0,0, \ldots)$, hence $C = \langle x_1, x_0 \rangle = \frac12$.
Note that $\frac{1}{n+1} (x_0 + x_1 + \ldots +
x_n) \to \frac{1}{\sqrt{2}} (1,0,0,0, \ldots) =: x_\infty$ for $n \to \infty$ and we conclude that $\CH_\infty$ contains the canonical orthonormal basis of $\ell^2 (n \in \Zset: -1 \le n < \infty)$ and hence
$\CH_\infty = \CH$ which is strictly larger than the union of all the $\CH_k = \spn(x_0, \ldots,x_k)$.
If we want an augmentation $\CH_{-1}$ then we need both $\CH_{-1} \subset \CH_0 = \Cset x_0$ and $\CH_{-1} \subset \CH^{\alpha_0} = \Cset x_\infty$ which forces $\CH_{-1} = \{0\}$.

Let us explicitly compute the innovation spaces:
\begin{align*}
\CD_{-1} & = \CH_{-1} = \{0\}, \\
\CD_0 & = \spn(x_9) = \Cset (1,1,0, \ldots) \\
\CD_1 & = \spn(x_0,x_1) \ominus \spn(x_0) = \Cset (1,-1,2,0,\ldots), \\
\CD_2 & = \spn(x_0,x_1,x_2) \ominus \spn(x_0,x_1) = \Cset (1,-1,-1,3,0,\ldots), \\
\ldots \\
\CD_k & = \spn(x_9,\ldots,x_k) \ominus \spn(x_0,\ldots,x_{k=1}) = \Cset (1,-1,\ldots,-1,k+1,0,\ldots). \\
\end{align*}
This SCH is not saturated and we can see that for $(1,-1,2,0,\ldots) \in \CD_1$ we have
\begin{align*}
& \alpha_0 (1,-1,2,0,\ldots) = (1,0,-1,2,0,\ldots) \\
&= \frac12 (1,1,0,\ldots) - \frac16 (1,-1,2,0,\ldots) + \frac23 (1,-1,-1,3,0,\ldots) \in \CD_0 \oplus \CD_1 \oplus \CD_2 
\end{align*}
with nontrivial components in all three innovation spaces.
\\

Back to the general case of a spreadable sequence $(\iota_n)_{n \ge 0}$ of isometries from $\CK$ into $\CH$. By making use of the saturation we can give a structure theorem for these which is actually a kind of corollary to the 'toy de Finetti' Theorem \ref{T:saturation} and shows that spreadable sequences (with respect to an inner product) are 'conditionally orthogonal' (in the sense explained below), completely analogous to the well known statements about conditional independence of spreadable random variables in the probabilistic de Finetti theorem. For some background on this see \cite{Ko10} and \cite{EGK17}. 

\begin{theorem}\label{C}
Let $(\iota_n)_{n \ge 0}$ be a spreadable sequence of isometries from $\CK$ into $\CH$
and let $Q$ be the orthogonal projection from $\CH_\infty$ to $\CH_\infty^{\alpha_0}$ (for the corresponding SCH). 
Then $\big( (\1 - Q) \iota_n(\CK) \big)_{n \in \Nset_0}$ is a sequence of orthogonal spaces.
\\
For the operator angle we have $C = \iota_1^* Q \iota_0 = \iota_0^* Q \iota_0 \ge 0$.
If the corresponding minimal SCH is saturated then this is an orthogonal projection. In general all positive contractions $C: \CK \rightarrow \CK$ can appear as operator angles. 
\end{theorem}

\begin{proof}
For the corresponding minimal SCH we have
\[
\iota_0(\CK) = \CH_0 \subset \hat{\CH}_0 = \hat{\CD}_{-1} \oplus \hat{\CD}_0,
\]
where the hat indicates the saturation. Note that $\hat{\CD}_{-1} = \CH_\infty^{\alpha_0}$
and $Q$ is the orthogonal projection to this subspace. Hence $(\1 - Q) \iota_0(\CK) \subset \hat{\CD}_0$. 
Hence if for $ x \in \CK$ we write $\iota_0(x) = x_{-1} \oplus x_0 \in \hat{\CD}_{-1} \oplus \hat{\CD}_0$ then
\[
\iota_n(x) = \alpha^n_0 \iota_0(x) = x_{-1} \oplus \alpha^n_0(x_0) \in \hat{\CD}_{-1} \oplus \hat{\CD}_n, 
\]
by Theorem \ref{T:saturation}
. So
\[
(\1 - Q) \iota_n(\CK) \subset \hat{\CD}_n.
\]
This proves the orthogonality of this sequence of subspaces. 
For $C = \iota^*_1 \iota_0$ we find
\[
\langle x, Cx \rangle = \langle \iota_1(x), \iota_0(x) \rangle = \langle x_{-1}, x_{-1} \rangle
= \langle \iota_1(x), Q \iota_0(x) \rangle = \langle \iota_0(x), Q \iota_0(x) \rangle
\]
and we find $C = \iota^*_1 Q \iota_0 = \iota^*_0 Q \iota_0 \ge 0$.

In the saturated case $\hat{\CD}_{-1}$ is contained in $\CH_0\; (= \hat{\CH}_0$ in this case) and so $C$ is nothing but the orthogonal projection onto
$\iota_0^* (\hat{\CD}_{-1})$.

Finally let us start with any contraction $C \ge 0$ on $\CK$ and construct a spreadable sequence with $C$ as the operator angle. Let $\CH := \bigoplus_{k=-1}^\infty \CK^{(k)}$, an infinite direct sum of copies of $\CK$, further define (for all $n \in \Nset_0$)
\begin{align*}
\iota_n: \CK & \rightarrow \CK^{(-1)} \;\oplus\; \CK^{(n)} \quad \subset \CH \\
x & \mapsto \sqrt{C}\,x \;\oplus\; \sqrt{\1-C}\,x.
\end{align*}
It can be verified that indeed $\iota_j^* \iota_i = C$ whenever $i \not= j$.
\end{proof}

\begin{corollary}
For a spreadable sequence $(x_n)_{n \ge 0}$ of unit vectors in a Hilbert space there exists $0 \le c \le 1$ such that $\langle x_j, x_i \rangle = c$ whenever $i \not= j$, and all such $c$ can be realized.
\end{corollary}

The contraction $C: \CK \rightarrow \CK$ is a complete invariant for spreadable sequences (wrt an inner product) in the following sense: Two spreadable sequences $(\iota_{i})_{i=0}^\infty \colon \CK \to \CH$ with contraction $C$ and 
$(\tilde{\iota}_{i})_{i=0}^\infty \colon \tilde{\CK} \to \tilde{\CH}$ with contraction $\tilde{C}$ are unitarily equivalent (that is, there exists an intertwining unitary $U_\infty \colon \CH_\infty \to \tilde{\CH}_\infty$ between the $\iota_n$ and $\tilde{\iota}_n$ and hence between the minimal SCHs)
if and only if $\CK$ and $\tilde{\CK}$ have the same dimension and, for the unitary $U \colon \CK \to \tilde{\CK}$ determined by $\tilde{\iota}_0 U = \iota_0$, we have $\tilde{C} U = U C$. In fact, by Theorem \ref{C} the contraction $C$ determines all the relevant inner products between the $\iota_n(\CK)$. 

\subsection*{Acknowledgements}
Conversations with Edwin Beggs at the Isaac Newton Institute Operator Algebras: subfactors and applications programme in 2017 inspired the analysis of the cohomology of SCHs. We thank both Edwin and the organisers of the programme. We are also grateful to Zachary Bufton 
for joint discussions.

\end{document}